\documentclass[11pt,leqno]{article}
\usepackage{latexsym}
\usepackage{graphics}
\usepackage{amsmath}
\usepackage{amsthm}
\usepackage{xspace}
\usepackage{amssymb}
\usepackage{color}
\usepackage{cite}

\usepackage{latexsym}
\usepackage{graphics}
\usepackage{amsmath}
\usepackage{amsthm}
\usepackage{xspace}
\usepackage{amssymb}
\usepackage{epsfig}

\def\ignore#1{{}}

\newcommand{\bea}{\begin{align*}}
\newcommand{\eaa}{\end{align*}}

\newcommand{\pr}{\mathbb{P}}
\renewcommand{\Pr}{\mathbb{P}}
\newcommand{\E}{\mathbb{E}}
\newcommand{\expect}{\mathbb{E}}
\newcommand{\rR}{\mathbb{R}}
\newcommand{\rN}{\mathbb{N}}
\newcommand{\rZ}{\mathbb{Z}}

\renewcommand{\epsilon}{\varepsilon}

\newcommand{\mb}[1]{\mbox{\boldmath $#1$}}

\newcommand{\ind}[1]{1{\left\{#1\right\}}}
\newcommand{\vct}[1]{{\mb #1}}
\newcommand{\D}{J}
\newcommand{\fM}{V}

\definecolor{Red}{rgb}{1,0,0}
\definecolor{Blue}{rgb}{0,0,1}
\definecolor{Green}{rgb}{0,1,0}

\newtheorem{theorem}{Theorem}
\newtheorem{lemma}[theorem]{Lemma}

\newtheorem{prop}{Proposition}
\newtheorem{coro}{Corollary}

\newtheorem{Defi}{Definition}

\theoremstyle{remark}
\newtheorem{remark}{Remark}

\title{Steady-state analysis of a multi-server queue\\ in the Halfin-Whitt regime}

\author{David Gamarnik$^a$\thanks{Support from NSF grant CMMI-0726733 is gratefully acknowledged.} \ \ \ \ \ \ Petar Mom\v{c}ilovi\'c$^b$\\
{\footnotesize \hspace{-0.15in}
$^a$Operations Research Center and Sloan School of Management, MIT, Cambridge, MA 02139; gamarnik@mit.edu}\\
{\footnotesize $^b$
EECS Department, University of Michigan, Ann Arbor, MI 48109; petar@eecs.umich.edu}
}

\date{}

\begin{document}

\maketitle

\begin{abstract}
We consider a multi-server queue in the Halfin-Whitt 
regime: as the number of servers~$n$ grows without a bound, the utilization approaches~$1$ from below at the rate $\Theta(1/\sqrt{n})$.
Assuming that the service time distribution is lattice-valued with a finite support, we characterize the limiting stationary queue length distribution in terms of the
stationary distribution of an explicitly constructed Markov chain.
Furthermore, we obtain an explicit
expression for the critical exponent for the moment generating
function of a limiting (scaled) steady-state queue length. This
exponent has a compact representation in terms of three parameters: the amount of spare capacity and the coefficients of variation of interarrival and service times. Interestingly, it matches an analogous exponent corresponding to a single-server queue
in the conventional heavy-traffic regime. \\ 


\end{abstract}


\section{Introduction}
In their seminal paper~\cite{HaW81} Halfin and Whitt formally introduced an
unconventional heavy traffic regime for queueing models. Unlike the traditional heavy traffic approach, in their regime (dubbed thereafter the Halfin-Whitt regime) high utilization is achieved by simultaneously increasing the arrival rate \emph{and} the number of servers $n$.
This regime is also reffered to as Quality- and Efficiency-Driven (QED) since it balances between the system utilization and quality of service perceived by customers; the steady-state queue length and waiting time scale respectively as $O(\sqrt{n})$ and $O(1/\sqrt{n})$, in some appropriate sense~\cite{HaW81}. Moreover, the QED regime can be understood as critical with respect to
the probability of wait, i.e., the limiting stationary probability of wait is {\em strictly} in $(0,1)$ in QED systems (the probabilities of wait~$0$ and~$1$ correspond to the quality-driven and efficiency-driven regimes, respectively). It should be noted that the QED regime was considered by Erlang~\cite{Erl48} in the context of numerical steady-state analysis of M/M/n and M/M/n/n systems. An asymptotic analysis of the closely related Erlang loss function was carried out in~\cite{Jag74}. A formal analysis of a queue with exponential service times in the QED regime was completed in~\cite{HaW81} by Halfin and Whitt. They established the criticality of the probability of wait in terms of the square-root spare capacity rule, both in steady-state and transient regimes.

Queueing models in the QED regime have found applications primarily in the area of large-scale call and customer contact centers~\cite{GKM03,AAM07}. Hence, a number of related models have been considered in the literature. Models with customer impatience relevant to call center management were studied in~\cite{FSS94,GMR01,ZeM05}. Approximations that take into account
finiteness of buffers were introduced in~\cite{Whi02a,Whi02b}. Revenue maximization and constraint satisfaction were considered in~\cite{ArM01,ArM02,BMR00,MaZ01,MZe06}. Optimal stochastic control of QED queues in various settings was examined in~\cite{Ata05b,Ata05,HZe05,Tez05}. The problem of joint control and staffing was studied in~\cite{AMR04,GAM06}. Most of the aforementioned results assume exponential service times.
This assumption significantly simplifies the analysis as one does not need to keep a track of residual service times. 
The literature on non-exponential service time distribution is limited. Phase-type service time distribution in the transient regime was considered  in Puhalskii and Reiman~\cite{PuR00}.
The case of deterministic service times in the steady-state regime was  considered
in Jelenkovi\'c et al.~\cite{JMM04}. A more recent work by Mandelbaum and Mom\v{c}ilovi\'c~\cite{MMo07} deals with the transient distribution of the virtual waiting time in the case of discrete service times with a finite support. A process-level limit for the G/GI/n queue for the case of general service time distributions was obtained recently by Reed~\cite{Ree07}.

In this paper we examine the stationary  behavior of a GI/GI/n
system in the Halfin-Whitt regime when the service times are
lattice-valued and the support is finite. More specifically, we
consider a sequence of first-come first-served queues indexed by the
number of servers $n \to \infty$. The utilization in the $n$th
system is $1-\beta/\sqrt{n}+o(1/\sqrt{n})$ for some parameter
$\beta>0$; equivalently, the number of servers $n$ is $R_n +
\beta\sqrt{R_n} + o(\sqrt{R_n)}$, where $R_n$ is the offered load of
the $n$th system. The service distribution does not change with $n$.
The stationary number of customers and waiting time in the $n$th
system are denoted by~$Q^n$ and~$W^n$, respectively. The first main result of the paper states the existence of limiting random variables
$\hat Q$ and $\hat W$ such that $Q^n/\sqrt{n} \Rightarrow \hat Q$
and $\sqrt{n} W^n \Rightarrow \hat W$, as $n\to\infty$. The distribution
of $\hat Q$ is shown to correspond to the unique stationary distribution of some underlying continuous-state Markov chain $\{(\hat
Q_t,\hat{\vct L}_t), t \in \rZ_+\}$, where $\{\hat{\vct L}_t,
t\in\rZ_+\}$ is limiting process corresponding to the vector of
customers in different stages of service. Our second main result
identifies the exact exponential decay rate of the limiting variable
$\hat Q$. Informally, we show that
 $\pr[\hat Q>x] \approx \exp\{-{2\beta x}/{(c_a^2+c_s^2)}\}$ for large~$x$,
where $c_a$ is the (limiting) coefficient of variation of
interarrival times and $c_s$  is the coefficient of variation of
service times. Our analysis uses quadratic and geometric Lyapunov functions to establish the
tightness of sequences $\{Q^n/\sqrt{n}, n\geq 1\}$ and
$\{\sqrt{n}W^n, n\geq 1\}$.

Next we list some notational conventions used  throughout the paper.
For two vectors $\vct x$ and~$\vct y$ with elements $x_i$ and $y_i$,
respectively, ${\vct x}\cdot{\vct y}$ denotes the dot product
$\sum_i x_i y_i$. All considered vectors are row vectors, and
transposition of a vector $\vct x$  is denoted by ${\vct x}^T$. Let
${\vct K} \triangleq (1,2,\ldots,K)$. For $\vct x \in \rR^m$,
$\|\vct x\|$ denotes the $L_1$-norm: $\|\vct x\|=\sum_{i=1}^m
|x_i|$. Denote by ${\cal T}:\rR^K\rightarrow \rR^K$ a linear
operator defined by
\[
{\cal T}\{(x_1,\ldots,x_K)\} = (x_2,\ldots,x_K,0).
\]
For $\rR^k$-valued random variables $\Rightarrow$ denotes the
convergence in distribution.  Given a random variable (r.v.) $X \in\rR$,
its moment generating function is $M_X(\theta)\triangleq \E
e^{\theta X}$. For every $\theta>0$, we denote by
$\mathcal{M}_\theta$ the family of sequences of r.v.s
$\{X^n,n\ge 1\}$ such that $\limsup_{n\rightarrow \infty}
M_{X^n}(\theta)<\infty$; let $\mathcal{M}_\infty=\cap_{\theta>0}\mathcal{M}_\theta.$ Given a
r.v. $X$, we write $X\in\mathcal{M}_\theta$
($X\in \mathcal{M}_\infty)$ if $\E e^{\theta X}<\infty$ ($\E e^{\theta X} <\infty$ for every $\theta>0$). We denote by $\E_\pi[\cdot]$ the expectation operator with respect to a probability measure $\pi$; similarly, we use $\Pr_{\pi}[\cdot]$ when the probability measure $\pi$ is not clear from the context. For two reals $x,y$ we
set $x \wedge y = \min \{x,y\}$, $x \vee y = \max\{x,y\}$, $x^+ = x
\vee 0$ and $x^- = (-x)^+$; when the argument of a unary operation
is a vector or  matrix it is understood that the operator is applied
element-wise. Symbols $\rZ_+$ and $\rR_+$ denote nonnegative
integers and reals, respectively.

The paper is organized as follows. In the next section we describe
the considered model and formally  introduce the Halfin-Whitt (QED) regime. Our
main results are stated in Section~\ref{section:MainResults}. Section~\ref{sec:prelim} contains preliminary results. The proofs of the main results can be found in
Sections~\ref{sec:Thm1}, \ref{sec:Thm2}, \ref{sec:thm3} and~\ref{sec:CorollaryWaiting}.

\section{Model}

\subsection{Queueing system description} \label{sec:GGN}

We consider a sequence of first-come first-served 
queues indexed by the number of servers~$n$. The details of our model are as follows.

\vspace{.025in}
{\em Service times.} Service times are  independent and identically
distributed (i.i.d.) r.v.s, equal in distribution to a r.v. $S$
that does not depend on $n$ and takes values in a finite set $\{s_1,
\ldots, s_K\}\subset \mathbb{R}_+$.
It is assumed that the set of service time values has a common
divisor $s>0$, i.e., $s_i = k_i s$ for some $k_i\in\mathbb{N}$,
$1\leq i\leq K$. Under this assumption, without loss of generality,
we adopt $s=1$ to be the largest common divisor of service time
values. Let $p_i \triangleq \Pr[S=i], 0\le i\le K$, where $K$ is the largest index such that $p_K>0$.
We assume $p_0=0$, that is no instantaneous service is possible.
Then the expected service time is
$\mu^{-1} \triangleq \E S = \sum_{i=1}^K ip_i$; the variance of $S$
is denoted by $\sigma_s$ and the coefficient of variation by
$c_s=\mu\sigma_s$. The steady-state behavior of the system with deterministic service times ($S=1$) has been characterized in~\cite{JMM04} and, thus, we
consider $\sigma_s>0$. In this case there exist two values of the
service time that are relatively prime, i.e., $p_i p_j >0$ for some relatively prime $i\not=j$; otherwise a simple time change argument can be applied to re-scale service times. For convenience let $\vct
p=(p_1,\ldots,p_K)$ and $\tilde{\vct p}=(\tilde p_1,\ldots,\tilde
p_K)$, where $\tilde p_i=\Pr[S\ge i]=\sum_{j\ge i}p_i$ describes the
tail of the service time distribution.

\vspace{.025in} {\em Arrival times.} Customers arrive to the $n$th
system according to a stationary renewal process with interarrival
times equal in distribution to $\tau_n$. The arrival rate $\lambda_n \triangleq 1/ \expect \tau_n$ is such that $\lambda_n \to \infty$ as $n \to \infty$ while the
coefficient of variation $c_{a,n}$ of interarrival times satisfies $c_{a,n}\to c_a$ as $n\to\infty$
for some $0\le c_a<\infty$. 
In view of the assumption $S \in\mathbb{N}$ ($s=1$), it is convenient to
define $A^n_t$, $t\in\rR$, as the number of arrivals in the time interval
$(t-1,t]$ in the $n$th system. In addition,
let $a^n_t$ denote the backward recurrence time of the arrival
process at time $t$, i.e., $a^n_t \triangleq \inf\{u>0:\, A^n_{t-u,t}>0\}$, where $A^n_{s,t}$ denotes the number of arrivals in the time interval $(s,t]$ for two reals $s<t$.
Our proving method is based on an
analysis of a time-embedded process that has a Markov property. Hence, we require that the arrival process has limited dependency in its structure. To this end, it is assumed that
the appropriately scaled number of arrivals, conditioned on the particular value of the
backward recurrence time~$a$, converges to a Gaussian distribution {\em uniformly} in~$a$, i.e., for every $t\in\rR$,
\begin{equation}
\sup_{a \geq 0} \, \left |\pr\left[\frac{A^n_{t} - \lambda_n}{\sqrt{\lambda_n}} \leq x \,\Bigg|\, a^n_{t-1}=a \right]
-\pr[A  \leq x]\right | \to 0, \label{eq:arrival-Adistr}
\end{equation}
as $n\to\infty$,
where $A$ is normally distributed with zero mean and variance $c_a^2$. Additionally we assume that (since convergence in distribution does not necessarily imply the convergence of moments) that
\begin{equation}
\sup_{a \geq 0} \, \expect\left[\frac{A^n_{t} -
\lambda_n}{\sqrt{\lambda_n}} \,\Bigg|\, a^n_{t-1}=a \right]  \to 0, \label{eq:arrival-A1}
\end{equation}
as $n\to\infty$, and
\begin{equation}
\limsup_{n\to\infty} \, \sup_{a \geq 0} \, \expect\left[\left(\frac{A^n_{t} - \lambda_n}
{\sqrt{\lambda_n}}\right)^2 \Bigg|\, a^n_{t-1}=a \right]  <\infty.
\label{eq:arrival-A2}
\end{equation}
There exists a broad class of arrival processes that satisfy these assumptions. The simplest one is the class of renewal processes with interarrival times that have uniformly in $a^n_t=a$ bounded conditional second moments. For example, let $\{\zeta_i, i \in \rZ\}$ be an i.i.d. sequence of nonnegative r.v.s with unit mean and a finite second moment. By setting $\zeta_i/\lambda_n$ to be the $i$th interarrival time in the $n$th process we obtain a
process that satisfies the aforementioned assumptions due to the
Central Limit Theorem for renewal processes~\cite[p.~114]{Dur05} when $\lambda_n \to \infty$ as $n\to\infty$.

Finally, since we consider multi-server queues in their steady states, the
distribution of interarrival times should be such that the
stationary distributions of all considered quantities exists and
are unique (for all finite $n$). See comments at the beginning of Section~\ref{section:MainResults} and~\cite[Ch.~XII]{Asm03} for details.

\vspace{.025in}
{\em Quantities of interest.} The number of customers {\em awaiting service} in the $n$th
queue at time~$t$ is denoted by~$Q^n_t$ and the {\em total} number of
customers in the system is denoted by $Y_t^n$. The fact that $Y^n_t  = n + Q^n_t$ when
all servers are busy while $Q^n_t=0$ when at least one
server is idle renders
\begin{equation} \label{eq:queue-idle2}
Q^n_t = (Y^n_t - n)^+
\end{equation}
for every time instant $t$. Let $L^n_{t,k}$, $k=1,\ldots,K$, be
the number of customers in service with  remaining service times in
the interval $(k-1,k]$ at time $t$. Notation ${\vct L}_t^n =
(L^n_{t,1},\ldots,L^n_{t,K})$ renders $\|\vct L^n_t\| \leq n$, with
strict equality corresponding to the case when at least one server is idle. The
following identity then holds for all $t\in\rR_+$:
\begin{align}\label{eq:queue-idle}
Q^n_t(n-\|\vct L^n_t\|)=0.
\end{align}
Let $\D^n_{t,k}$, $k=1,\ldots,K$, be the number of customers with
service  requirement $k$ that enter service during the time interval
$(t-1,t]$; set $\vct \D^n_t=(\D^n_{t,1},\ldots,\D^n_{t,K})$. Thus,
$\|\vct \D^n_t\|$ is the total number of customers that enter
service during the time interval $(t-1,t]$ and
\begin{equation}
Q^n_{t+1}=Q^n_t+A^n_{t+1}-\|\vct \D^n_{t+1}\|. \label{eq:queue-idle3}
\end{equation}

\subsection{QED regime and scaling}
The offered load in the $n$th system is $\lambda_n/\mu$ and, hence, the  utilization is given by $\rho_n \triangleq\lambda_n/(n\mu)$. In the Halfin-Whitt (QED) regime the relationship between the utilization and number of servers satisfies
\begin{equation}\label{eq:HW}
\sqrt{n} (1-\rho_n) \to \beta,
\end{equation}
as $n \to \infty$, for some $\beta>0$, or  equivalently $n =
\lambda_n/\mu + \beta \sqrt{\lambda_n /\mu } +
o(\sqrt{\lambda_n/\mu})$ as $n \to \infty$. For notational
simplicity we let $\beta_n$ be a quantity satisfying
$n=\lambda_n/\mu + \beta_n \sqrt{n}$, i.e.,
\[
\beta_n = (n - \lambda_n/\mu)/\sqrt{n} \to \beta,
\]
as $n \to\infty$. Under such a scaling, the following centered and
scaled versions of r.v.s indexed by $t\in\rR_+$ 
are of interest:
\begin{align}
\hat A^n_t &\triangleq (A^n_t - \lambda_n)/\sqrt{n}, \nonumber \\
\hat Q^n_t &\triangleq  Q^n_t / \sqrt{n}, \nonumber \\ 
\hat{\vct L}^n_t &\triangleq (\vct L^n_t - \lambda_n \tilde {\vct p})/\sqrt{n}, \nonumber \\ 
\hat {\D}^n_t &\triangleq (\|\vct \D^n_t \|-\lambda_n) /\sqrt{n}, \label{eq:hatD}\\
\hat {\vct \D}^n_t &\triangleq (\vct \D^n_t - \|\vct \D^n_t \| \vct p ) /\sqrt{n}, \nonumber \\
\hat Y^n_t &\triangleq  \sum_{i=1}^K  \hat L^n_{t,i}  + \hat Q^n_t = (Y^n_t - \lambda_n/\mu)/\sqrt{n}. \label{eq:hatY}
\end{align}
Given these definitions, the counterparts of~(\ref{eq:queue-idle2}), (\ref{eq:queue-idle}) and (\ref{eq:queue-idle3}) are
\begin{align}
\hat Q^n_t &= (\hat Y^n_t - \beta_n)^+, \label{eq:Y=Q-beta}\\
\hat Q^n_t &\left(\beta_n - \sum_{k=1}^K \hat L^n_{t,k}\right) = 0 \nonumber
\end{align}
and
\begin{equation}
\hat Q^n_{t+1} =\hat Q^n_t+\hat A^n_{t+1}-\hat \D^n_{t+1}, \label{eq:scaledDynamics}
\end{equation}
respectively.

\section{Main results}\label{section:MainResults}

A multi-server queue can be described by the standard Kiefer-Wolfowitz vector~\cite{KWo55} of residual workloads, e.g.,
see~\cite{Asm03,BBr03}. Provided the stability condition
$\rho_n=\lambda_n/(n\mu)<1$ is satisfied and the arrival process is renewal, in~\cite{KWo55} it was
established that all relevant stationary measures exist when the
system is observed just before arrivals, i.e., stationary measures
exists for this particular time-embedded process. In order to ensure the
existence of stationary probabilities for continuous-time processes
$\{(Q^n_t, {\vct L}^n_t), t\in\rR_+\}$ and $\{W^n_t,
t\in\rR_+\}$ additional conditions are needed~\cite[p.~348]{Asm03}.
We assume that these stationary distributions exists and are unique. Let $\pi_n$
be the stationary probability law of $\{(\hat Q^n_t, \hat{\vct
L}^n_t), t\in\rR_+\}$, i.e., $\pi_n$ is time invariant with respect
to $t$. We characterize the limit of $\pi_n$ as $n \to \infty$
in terms of the stationary probability of a certain discrete-time process
$\{(\hat Q_t,\hat{\vct L}_t),\, t \in \rZ_+\}$. Although the
processes $\{(\hat Q^n_t,\hat {\vct L}^n_t), t \in\rR\}$ are inherently
continuous-time, for the purposes of characterizing their stationary
distributions it is sufficient to consider their time-embedded
versions ($t\in\rZ_+$ due to the lattice-valued nature of service
times $S\in\rN$). Such an approach has an advantage since these discrete-time
processes have a tractable Markovian structure that is amenable to the
Lyapunov function method~\cite{MTw93}.

Next we construct the Markov chain $\{(\hat Q_t,\hat{\vct L}_t),\, t
\in \rZ_+\}$  with state space $\rR^{K+1}$. To this end, let $\{\hat A_t, \, t\in
\rZ_+\}$ be an i.i.d. sequence of zero mean normal r.v.s
with variance $\mu c_a^2$. Also let $\{\hat{\vct \D}_t, \, t\in
\rZ_+\}$ be an i.i.d. sequence of normal random vectors with the zero
mean and covariance matrix $\mu\Sigma$, elements of $\Sigma$
defined by
\begin{align}\label{eq:Sigma}
\Sigma_{ij} = \begin{cases}
(1-p_i) p_i, & 1\leq i=j\leq K, \\
- p_i p_j, & 1\leq i \not= j\leq K;
\end{cases}
\end{align}
the sequences $\{\hat A_t,\, t\in\rZ\}$ and $\{\hat{\vct \D}_t,\, t\in\rZ\}$ are independent.
The process $\{(\hat Q_t,\hat{\vct L}_t), \,t\in\rZ_+\}$ is defined by the following three recursions
\begin{align}
\hat  {\vct  L}_{t+1} &= {\cal T}\{\hat{\vct  L}_t\} + \hat{\vct \D}_{t+1} + \hat \D_{t+1} {\vct  p}, \label{eq:L'} \\
\hat Q_{t+1} &=\left( \hat Q_t +\hat A_{t+1} + \sum_{k=2}^K \hat L_{t,k} - \beta \right)^+,
\label{eq:Q'} \\
\hat \D_{t+1}&=\left(\hat Q_t + \hat A_{t+1} \right) \wedge \left(\beta  - \sum_{k=2}^K \hat L_{t,k}\right), \label{eq:D'}
\end{align}
and an initial condition $(\hat Q_0,\hat{\vct L}_0)$ that is independent of
$\{\hat A_t, \, t\in \rZ_+\}$ and $\{\hat {\vct \D}_t, \, t\in \rZ_+\}$; the random vector $(\hat Q_0,\hat{\vct L}_0)$ satisfies $\sum_{k=1}^K \hat L_{0,k} \leq \beta$ and $\hat Q_0 (\sum_{k=1}^K \hat L_{0,k} - \beta)=0$ by definition. It is straightforward to verify that the preceding defines a continuous-state Markov chain due to the i.i.d. nature of $\{\hat A_t, \, t\in
\rZ_+\}$ and $\{\hat{\vct \D}_t, \, t\in \rZ_+\}$. Observe that (\ref{eq:L'}), (\ref{eq:Q'}) and (\ref{eq:D'}) imply, for all $t \in \rN$, $\sum_{k=1}^K \hat L_{t,k} \leq \beta$
and
\[
\hat Q_t \left(\sum_{k=1}^K \hat L_{t,k} - \beta\right) = 0.
\]
We define a process $\{\hat Y_t, t \in\rZ_+\}$ by $\hat Y_t = \sum_{k=1}^K \hat L_{t,k} + \hat Q_t$
and note that it satisfies $\hat Q_t = (\hat Y_t - \beta)^+$, $t\in\rZ_+$; we often refer
to this process as the limiting number of customers in the queue.

Our first main result states the existence of a distributional limit
of $(\hat Q^n, \hat{\vct L}^n)$, as $n\rightarrow\infty$, where the pair $(\hat Q^n, \hat{\vct L}^n)$ is distributed according to $\pi_n$. In
particular, we relate the sequence of stationary distributions
$\{\pi_n, n\geq 1\}$ of $\{(\hat Q^n_t, \hat {\vct L}^n_t), t \in \rR_+\}$ to
the stationary distribution of the discrete-time chain $\{(\hat Q_t, \hat {\vct L}_t), t \in \rZ_+\}$. The proof is based on a tightness
argument and can be found in Section~\ref{sec:Thm1}.

\begin{theorem}\label{theorem:MainResultTightness} $\pi_n \Rightarrow\pi_*$ as $n\to\infty$, where $\pi_*$ is the unique stationary distribution of the Markov chain $\{(\hat Q_t, \hat {\vct L}_t), t \in \rZ_+\}$.
\end{theorem}

\begin{proof}[Outline of the proof:] The proof consists of three parts:
(i) demonstrating that the sequence $\{(\hat Q^n_t, \hat {\vct L}^n_t), n \geq 1\}$
is tight with respect to the sequence of distributions $\{\pi_n, n \geq 1\}$ (as $n\to\infty$), (ii)
showing that the stationary distribution of $\{(\hat Q^n_t, \hat {\vct L}^n_t), t \in \rR_+\}$ converges
to a stationary distribution of $\{(\hat Q_t, \hat {\vct L}_t), t \in \rZ_+\}$ as $n\to\infty$,
and (iii) proving that $\{(\hat Q_t, \hat {\vct L}_t), t \in\rZ_+\}$ has a unique stationary distribution $\pi_*$.
We briefly outline the main argument for (i), as the proofs of (ii) and (iii) follow
more or less a standard argument.

A polynomial function $\Psi_\theta({\vct y}, {\vct z}) = (\tilde {\vct p} \cdot {\vct y} + {\vct \alpha} \cdot {\vct z})^\theta$ is defined, with ${\vct \alpha} \in \rR^{K^2}$ being fixed (Section~\ref{section:Lyapunov2}); function $\Psi_1$ can take negative values. For notational simplicity let $\bar {\vct Y}^n_t = (\hat Y^n_{t},\ldots,\hat Y^n_{t-K+1})$ and $\bar {\vct Z}^n_t = (\hat {\vct Z}^n_{t},\ldots,\hat {\vct Z}^n_{t-K+1})$, where $\hat {\vct Z}^n_t = \hat {\vct \D}^n_t + {\vct p} \hat A^n_t$. Based on preliminary results (Sections~\ref{sec:tep} and~\ref{section:NumberSystem}) the following is derived (Section~\ref{section:Lyapunov2}) for some set ${\cal R}^n$:
\[
\expect\left[\Psi_2(\bar {\vct Y}^n_{t}, \bar {\vct Z}^n_{t}) \ind{\bar {\vct Y}^n_{t-1} \not\in {\cal R}^n} - \Psi_2(\bar {\vct Y}^n_{t-1}, \bar {\vct Z}^n_{t-1}) | \bar {\vct Y}^n_{t-1}, \bar {\vct Z}^n_{t-1} \right] \leq -\delta \Psi_1(\bar {\vct Y}^n_{t-1}, \bar {\vct Z}^n_{t-1}) + \psi
\]
for some $\delta>0$, $\psi<\infty$ and all $n$ large enough (Proposition~\ref{prop:2LyapunovBound}), and
\[
\limsup_{n\to\infty} \expect_{\pi_n} [\Psi_2(\bar {\vct Y}^n_{t}, \bar {\vct Z}^n_{t}) \ind{\bar {\vct Y}^n_{t-1}, \bar {\vct Z}^n_{t-1} \in {\cal R}^n} ] < \infty
\]
(Lemma~\ref{lemma:inRbound}). These two relationships can be combined (Theorem~\ref{theorem:BoundsStationaryPolynomial}) to obtain
\[
\limsup_{n\to\infty} \expect_{\pi_n} [\Psi_1(\bar {\vct Y}^n_{t}, \bar {\vct Z}^n_{t})] < \infty.
\]
On the other hand, the expectation of the negative part of $\Psi_1(\bar {\vct Y}^n_{t}, \bar {\vct Z}^n_{t})$ is also bounded in the limit (Lemma~\ref{lemma:<0})
\[
\limsup_{n\to\infty} \expect_{\pi_n} [-\Psi_1(\bar {\vct Y}^n_{t}, \bar {\vct Z}^n_{t}) \ind{\Psi_1(\bar {\vct Y}^n_{t}, \bar {\vct Z}^n_{t})<0} ] < \infty.
\]
Finally, the tightness of $\{\hat Y^n_t, n \geq 1\}$ (and hence of $\{\hat Q^n_t, n \geq 1\}$ since $\hat Q^n_t = (\hat Y^n_t-\beta)^+$) with respect to stationary $\{\pi_n, n \geq 1\}$ is due to $\expect_{\pi_n} [\Psi_1(\bar {\vct Y}^n_{t}, \bar {\vct Z}^n_{t})] = \expect_{\pi_n}[\tilde {\vct p} \cdot \bar {\vct Y}^n_t ] = \expect_{\pi_n} \hat Y^n_t/\mu$.
\end{proof}

Let $(\hat Q, \hat{\vct L})$ be distributed according to $\pi_*$, i.e., if $Q^n$ is the stationary number of customers in the $n$th queue, then $Q^n/\sqrt{n} \Rightarrow \hat Q$ as $n\to\infty$. It is immediate that $\pr[\hat Q=0] \in (0,1)$ since the Gaussian term $\hat A_t$ in (\ref{eq:L'}), (\ref{eq:Q'}) and (\ref{eq:D'}) has infinite support.
The convergence $\pi_n\Rightarrow \pi_*$ implies $\Pr[\hat Q^n=0] \to \pr[\hat Q=0]\in (0,1)$ as $n\to\infty$, and, thus, the system is indeed in the QED regime.

Our second result establishes the critical exponent for the moment generating function of~$\hat Q$. The proof can be found in Section~\ref{sec:Thm2}.

\begin{theorem}\label{theorem:MainResultDecayRate}
Let $\theta^*= 2\beta/(c_a^2+c_s^2)$. Then $\E e^{\theta \hat Q} <
\infty$ if  $\theta<\theta^*$ and $\E e^{\theta \hat Q} = \infty$ if
$\theta>\theta^*$.
\end{theorem}

\begin{proof}[Outline of the proof:] Here we outline just the proof of the
statement $\expect e^{\theta \hat Q} < \infty$ if $\theta<\theta^*$.
The key idea is to define a geometric Lyapunov function
$\Phi_{\theta}({\vct y}, {\vct z}) = \exp\left\{\theta \tilde {\vct
p} \cdot {\vct y}  + \theta {\vct \alpha} \cdot {\vct z}\right\}$
(Section~\ref{section:LyapunovFunction}) with ${\vct \alpha} \in
\rR^{K^2}$ being fixed. Based on the rules according to which the
number of customers in the system evolves
(Section~\ref{section:NumberSystem}) it is possible to define a set
$\cal R$ (Section~\ref{section:LyapunovFunction}) such that
\[
\expect \left[\Phi_\theta(\bar {\vct Y}_{t}, \bar {\vct Z}_{t}) \ind{\bar {\vct Y}_{t-1} \not\in {\cal R}} | \bar {\vct Y}_{t-1}, \bar {\vct Z}_{t-1} \right] \leq (1-\delta) \Phi_\theta(\bar {\vct Y}_{t-1}, \bar {\vct Z}_{t-1})
\]
for all $\theta < \theta^*/\mu$ and some $\delta<1$ (Proposition~\ref{prop:LyapunovBoundLimit}), and $\expect_{\pi_*} [\Phi_\theta(\bar {\vct Y}_{t}, \bar {\vct Z}_{t}) 1\{\bar {\vct Y}_{t-1} \in {\cal R}\}] < \infty$ for $\theta>0$ (Lemma~\ref{lemma:inRboundLimit}). The preceding two inequalities are combined to conclude $\expect_{\pi_*} \Phi_\theta(\bar {\vct Y}_{t}, \bar {\vct Z}_{t}) < \infty$ for $\theta<\theta^*/\mu$ (Theorem~\ref{theorem:BoundsStationaryGeometric}) and
\[
\expect_{\pi_*} \left[\exp\left\{ \theta {\vct p} \cdot \bar {\vct Y}_t \right\} \right] < \infty
\]
follows since $\hat {\vct J}_t$ and $\hat A_t$ are normally
distributed by definition.  Finally, the proof is concluded by
showing that
\[
\expect_{\pi_*} \left[\exp\left\{\theta | \hat Y_t/\mu - {\vct p} \cdot \bar {\vct Y}_t | \right\}\right] < \infty
\]
for all $\theta>0$.
\end{proof}

The theorem is stated for the limiting queue length $\hat Q$. With additional conditions  on the arrival processes a weaker result can be obtained for the pre-limit variables $\hat Q^n$ by only slightly modifying the proof of Theorem~\ref{theorem:MainResultDecayRate}. Namely, it is needed that the arrival process satisfies
\begin{equation}
\limsup_{n\to\infty} \sup_{a\ge 0}\,  \expect \left[e^{\theta |\hat A^n_t|} \,\big|\, a^n_{t-1}=a\right] \le e^{c\theta^2} \label{eq:arrival-Aexp}
\end{equation}
for some $c<\infty$ and for every $\theta>0$.
Then, the following result is established by exploiting the preceding relationship.

\begin{theorem}\label{theorem:MainResultExpErg}
Suppose that~(\ref{eq:arrival-Aexp}) holds. There exists $\theta>0$ such that
\[
\limsup_{n\to\infty} \expect e^{\theta Q^n/\sqrt{n}} < \infty.
\]
\end{theorem}

\begin{proof}
See Section~\ref{sec:thm3}.
\end{proof}

We conjecture that the threshold value of $\theta$ in Theorem~\ref{theorem:MainResultExpErg} is given by $\theta^*$ (defined in Theorem~\ref{theorem:MainResultDecayRate}). We remark that the criticality of the exponent $\theta^* = 2\beta/(c_a^2+c_s^2)$ is consistent with the results obtained earlier
in~\cite{HaW81,JMM04}. Namely, in the GI/M/n queue in the QED regime
the conditional limited scaled steady-state number of customers is
exponentially distributed~\cite{HaW81}
\[
x^{-1} \log \Pr[\hat Q > x | \hat Q > 0] = -2\beta/(c_a^2+1),
\]
$x>0$, while for the QED GI/D/n queue~\cite{JMM04} one has, as $x\to\infty$,
\[
x^{-1} \log \Pr[\hat Q > x | \hat Q > 0] \to -2\beta/c_a^2;
\]
recall that in both cases $\Pr[\hat Q>0] \in (0,1)$ for every
$\beta>0$. Furthermore,  we point out that the same exponent $\theta^*$
appears in the Kingman approximation~\cite{Kin61,Kin64} for a
single-server queue in the conventional heavy-traffic regime.
Moreover, the same exponent was established in analyses
of queues with a fixed number of servers in the same heavy-traffic regime.

In particular, consider a sequence of {\em single}-server queues indexed by $n$. The arrival
rate to the the $n$th system is $\lambda_n \to \infty$, with the
arrival process being renewal, satisfying the Central Limit Theorem and $c_{a,n} \to c_a$ as $n\to\infty$.
The service times of customers are i.i.d. and equal in distribution
to $S/n$ (equivalently, the service capacity grows linearly in $n$), and, thus, the utilization is given by $\rho_n =
\lambda_n/(n\mu)$. Let $\tilde Q^n$ be the steady-state number of
customers {\em awaiting service} in the system indexed by $n$; $\tilde Q^n$ and the total
number of customers in the system differ by at most one at any point
in time. If $\sqrt{n}(1-\rho_n) \to \beta>0$, as $n\to\infty$, then $\tilde Q^n/\sqrt{n} \Rightarrow \tilde Q$, as
$n\to\infty$, where $\tilde Q$ is exponentially
distributed~\cite[Sect.~9.6]{Whi02c} (see
also~\cite[Sect.~5.7]{Whi02c}):
\[
x^{-1} \log \Pr[\tilde Q > x] = -\theta^*,
\]
$x>0$, where $\theta^*$ is as in
Theorem~\ref{theorem:MainResultDecayRate}. The agreement of the critical exponent $\theta^*$ in the corresponding single- and $n$-server ($n\to\infty$)
systems is interesting since the two evolve under different rules.
Observe that, as $n\to\infty$, the total number of customers in the
single-server system is $\Theta(\sqrt{n})$, as $n\to\infty$, while for the $n$-server system that quantity is $\Theta(n)$, as $n\to\infty$.

In conclusion of this section we obtain an analogue of Theorem~\ref{theorem:MainResultDecayRate} for waiting times using the Distributional Little's Law~\cite{HNe71} applied to the waiting room. In the case of process-level (transient) analysis a result of~\cite{Puh94} (see also~\cite[Lemma~A.2]{PuR00}) is utilized typically to make a ``translation" between
the queue length and waiting time processes. Here we provide a simple independent proof for the stationary waiting time~$W^n$.

\begin{coro}\label{coro:Waiting}
If $\hat W\triangleq\mu^{-1}\hat Q$ then $\sqrt{n} W^n \Rightarrow \hat W$ as $n\to\infty$. Consequently,
$\expect e^{\theta \hat W}<\infty$ if $\theta<\mu\theta^*$ and $\expect e^{\theta \hat W}=\infty$ if $\theta>\mu\theta^*$.
\end{coro}

\begin{proof}
See Section~\ref{sec:CorollaryWaiting}.
\end{proof}

\section{Preliminary results} \label{sec:prelim}
The section contains 4 subsections. In the first subsection we consider a time-embedded version of $\{(\hat Q^n_t, {\vct L}^n_t), t \in\rR_+\}$. The number of customers in the finite-$n$ and limiting systems is considered in the second subsection. Quadratic and geometric Lyaponov functions are introduced and analyzed in the last two subsections.

\subsection{Time-embedded process} \label{sec:tep}
In  this section we examine the  triple $(Q^n_t, {\vct L}^n_t, a_t)
\in \rZ_+^{K+1} \times \rR_+$ and the laws governing its evolution
in time. The process $\{(Q^n_t, {\vct  L}^n_t, a^n_t),\,
t\in\rR_+\}$ is not Markovian due to the non-exponential nature of
service times. Hence, in order to avoid enlarging the state space,
we consider its time-embedded version $\{(Q^n_t, {\vct  L}^n_t,
a^n_t),\, t\in\rZ_+\}$, i.e., the original process observed at
discrete-time instances $t \in\rZ_+$ (recall from Section~\ref{sec:GGN} that $S\in\rN$). As seen in the following
proposition, the evolution of the later process is determined by the
number of arrivals ($A^n_t$) and customers that enter service ($\vct \D^n_t$) during a unit time interval.

\begin{prop}\label{prop:MC}
 The process $\{(Q^n_t,{\vct  L}^n_t,a^n_t), \, t \in \rZ_+\}$ is a Markov chain.
For every $t \in \rZ_+$ the value of $(Q^n_{t+1},{\vct L}^n_{t+1})$ satisfies
\begin{align}
{\vct L}^n_{t+1} &= {\cal T}\{{\vct L}^n_t\} + {\vct \D}^n_{t+1}, \label{eq:L} \\
Q^n_{t+1} &= \left(Q^n_t + A^n_{t+1} + \| {\vct L}^n_{t}\| - n - L^n_{t,1}\right)^+, \label{eq:Q}
\end{align}
where ${\vct \D}^n_{t+1} \in \rZ_+^K$ is a multinomially distributed random vector that obeys
\begin{equation}
\|{\vct  \D}^n_{t+1}\| = \left(Q^n_t+A^n_{t+1}\right) \wedge (n - \| L^n_{t,k}\| + L^n_{t,1}) \label{eq:D}
\end{equation}
and
\begin{equation}
\expect [{\vct \D}^n_{t+1} \, |\, \|{\vct \D}^n_{t+1}\|] = \|{\vct \D}^n_{t+1}\| \, {\vct p}; \label{eq:D2}
\end{equation}
given $\|{\vct \D}^n_{t+1}\|$ vector ${\vct J}^n_{t+1}$ is
conditionally independent of $\{(Q^n_t,{\vct  L}^n_t,a^n_t), \, t
\in \rZ_+\}$.
\end{prop}

\begin{proof} It is sufficient to demonstrate (\ref{eq:L}), (\ref{eq:Q}) and (\ref{eq:D}); equality~(\ref{eq:D2}) is a straightforward consequence of the i.i.d. nature of service times. The Markov property follows from these relationships and the renewal structure
of the arrival process.

Consider the number of customers that enter service in the time
interval $(t,t+1]$. At time~$t$ there are $\|{\vct  L}^n_t\|$
customers in service by the definition of ${\vct  L}^n_t$. Out of
these $\|{\vct L}^n_t\|$ customers, $L^n_{t,1}$ depart from the
system not later than time $(t+1)$ since their residual service
requirements at time $t$ are at most~1 (by the definition of
$L^n_{t,1}$). This yields that $n - \|\vct L^n_t\|  + L^n_{t,1}$
customers can potentially enter service in the time interval
$(t,t+1]$; recall that $n - \|{\vct L}^n_t\| $ is the number of
idle servers at time $t$. On the other hand, the number of customers
that can enter service in $(t,t+1]$ is at most $Q^n_t+A^n_{t+1}$.
Thus, the number of customers that enter service in $(t,t+1]$ is
\begin{equation}
\|\vct \D^n_{t+1}\|=(Q^n_t+A^n_{t+1}) \wedge (n - \|{\vct L}^n_t\|  + L^n_{t,1} ), \label{eq:enterS}
\end{equation}
rendering (\ref{eq:D}). Now, customers in service at time $t+1$
with residual service requirements in $(i-1,i]$ are of two types:
(i) customers already in service at time $t$, and (ii) customers that
do enter service in $(t,t+1]$. Thus, formally
\begin{equation}
L^n_{t+1,i} =
\begin{cases}
L^n_{t,i+1} + \D^n_{t+1,i}, & i=1,\ldots,K-1,\\
\D^n_{t+1,i}, & i=K.
\end{cases} \label{eq:vctL}
\end{equation}
The multinomial distribution of ${\vct \D}^n_{t+1}$ follows from the
assumption that customers' service requirements are i.i.d. r.v.s, independent from the arrival processes.  Rewriting~(\ref{eq:vctL}) in a vector form renders~(\ref{eq:L}).

In order to establish the value of $Q^n_{t+1}$,  it is sufficient to
consider the difference between the number of customers that could
start receiving service in the time interval $(t,t+1]$ and the
actual number of customers that enter service,
i.e.,~(\ref{eq:queue-idle3}) and (\ref{eq:enterS}) yield
\begin{align*}
Q^n_{t+1} &= (Q^n_t+A^n_{t+1}) - \|\vct J^n_{t+1} \| \\
&= (Q^n_t + A^n_{t+1} + \|{\vct L}^n_t\| - L^n_{t,1} - n)^+,
\end{align*}
and (\ref{eq:Q}) holds. This concludes the proof.
\end{proof}

An analogue of Proposition~\ref{prop:MC} for scaled processes is stated next.

\begin{coro}\label{coro:MCscaled}
The process $\{(\hat Q^n_t, \hat {\vct  L}^n_t, a^n_t), \, t \in \rZ_+\}$ is a Markov chain and it satisfies
\begin{align}
\hat {\vct  L}^n_{t+1} &= {\cal T}\{\hat {\vct  L}^n_t\}
+ \hat {\vct \D}^n_{t+1}+\hat \D^n_{t+1}{\vct p}, \label{eq:Lnhat} \\
\hat Q^n_{t+1} &= \left(\hat Q^n_t + \hat A^n_{t+1} + \sum_{i=2}^K \hat L^n_{t,i} - \beta_n \right)^+, \nonumber \\ 
\hat \D^n_{t+1} &= \left(\hat Q^n_t+\hat A^n_{t+1}\right) \wedge \left(\beta_n - \sum_{i=2}^K \hat L^n_{t,i} \right), \nonumber 
\end{align}
where $\hat{\vct \D}^n_{t+1}$ conditional on $\hat \D^n_{t+1}$ is independent of $\{(\hat Q^n_t, \hat {\vct  L}^n_t, a^n_t), \, t \in \rZ_+\}$.
\end{coro}

\begin{proof} The Markov property follows from Proposition~\ref{prop:MC} and the fact that there
exists a one-to-one mapping between $(\hat Q^n_t,\hat {\vct  L}^n_t)$ and $(Q^n_t,{\vct  L}^n_t)$.
Now, (\ref{eq:L}) implies
\begin{align*}
{\vct  L}^n_{t+1} - \lambda_n \tilde{\vct p} &= {\cal T}\{{\vct  L}^n_t\} + {\vct \D}^n_{t+1 } - \lambda_n \tilde{\vct p}\\
&= \left( {\cal T}\{{\vct L}^n_t\} - \lambda_n (\tilde{\vct p}-\vct p) \right) + \left({\vct \D}^n_{t+1} -\|{\vct \D}^n_{t+1}\| {\vct p}\right) + \left(\|{\vct \D}^n_{t+1} \| - \lambda_n \right) {\vct p}.
\end{align*}
This equality and the observation ${\cal T}\{{\vct  L}^n_t-\lambda_n \tilde{\vct p}\}={\cal T}\{{\vct  L}^n_t\} - \lambda_n (\tilde{\vct p}-\vct p)$ (due to the definition of~$\tilde {\vct p}$) yield~(\ref{eq:Lnhat}). The remaining relationships are obtained similarly from their counterparts (\ref{eq:Q}) and~(\ref{eq:D}).
\end{proof}

Properties of the vector $\hat {\vct J}^n_t$ are summarized in the following lemma.

\begin{lemma}\label{lemma:Dbound} Vector $\hat{\vct \D}^n_t$ satisfies for every $k=0,1,\ldots,n$
and $\theta>0$
\begin{align*}
\expect \left[\left( {\vct K} \cdot \hat {\vct \D}^n_t \right)^2 \,\bigg|\, \| {\vct \D}^n_t \|
= k \right] &= k\sigma^2_s/n, \\
\expect \left[\left( \hat \D^n_{t,j} \right)^2 \,\bigg|\, \| {\vct \D}^n_t \|
= k \right] &= k p_j(1-p_j)/n, \\
\expect \left[\exp\left\{\theta {\vct K} \cdot \hat {\vct \D}^n_t \right\} \,\bigg|\, \| {\vct \D}^n_t \|
= k \right] &= \left(\expect e^{\theta \frac{S - 1/\mu}{\sqrt{n}}} \right)^k \\
&\le \left(\expect e^{\theta \frac{S - 1/\mu}{\sqrt{n}}} \right)^n.
\end{align*}
\end{lemma}

\begin{proof} Let $\{S_i\}_{i=1}^k$ be a sequence of i.i.d. r.v.s equal in distribution to $S$. Then, the definition of $\hat {\vct \D}^n_t$ renders
\begin{align*}
\expect \left[\left( {\vct K} \cdot \hat {\vct \D}^n_t \right)^2  \,\bigg|\, \| {\vct \D}^n_t \| = k \right] &=
\expect \left(\sum_{i=1}^k \frac{S_i - 1/\mu}{\sqrt{n}} \right)^2 \\
&= k \sigma^2_s/n.
\end{align*}
The other two equalities are obtained in a similar straightforward fashion. The inequality is due to $\expect e^{\theta (S-1/\mu)/\sqrt{n}} \geq 1$. This follows from the convexity of $e^{\theta (x-1/\mu)/\sqrt{n}}$ in $x$ and Jensen's inequality.
\end{proof}

\subsection{Number in system}\label{section:NumberSystem}
This section is devoted to the detailed analysis of the rescaled number of customers in the system $\{\hat Y^n_t, t\in\rZ_+\}$ and its limiting counterpart. The dynamics of $\{\hat Y^n_t, t\in\rZ_+\}$ is related to a newly introduced process
\begin{equation}
\hat{\vct  Z}^n_t = (\hat Z^n_{t,1}, \ldots, \hat Z^n_{t,K}) \triangleq \hat {\vct  \D}^n_t + \vct  p \hat A^n_t \label{eq:defZ}
\end{equation}
and in particular to
\begin{align}
\hat \fM^n_t &\triangleq \sum_{i=1}^K \sum_{j=i}^K (\hat \D^n_{t+1-i,j} + p_j \hat A^n_{t+1-i}), \label{eq:hatMt} \\
&=\sum_{i=1}^K \sum_{j=i}^K \hat Z^n_{t+1-i,j}, \label{eq:MZ}
\end{align}
as stated in the next lemma. Informally, for large $n$, process $\{\hat\fM^n_t, t\in\rZ_+\}$ serves as a proxy for a scaled infinite-server process. We remark that the lemma is a discrete-time analogue of~(1.1) in~\cite{Ree07}.

\begin{lemma} \label{lemma:Yevo}
The process $\{\hat Y^n_t, t \in \rZ_+\}$ satisfies for all $t \geq K$
\[
\hat Y^n_t = \hat \fM^n_t + \sum_{i=1}^K p_i (\hat Y^n_{t-i} - \beta_n)^+.
\]
\end{lemma}

\begin{proof}
Equalities~(\ref{eq:Lnhat}) and~(\ref{eq:scaledDynamics}) yield an expression for the $K$th element of the vector $\hat {\vct  L}^n_{t+1}$:
\[
\hat {  L}^n_{t+1,K} = \hat \D^n_{t+1,K} + (\hat A^n_{t+1} + \hat Q^n_t - \hat Q^n_{t+1}) p_K.
\]
Furthermore, using~(\ref{eq:Lnhat}) iteratively it is straightforward to obtain the remaining elements of $\hat {\vct  L}^n_{t+1}$. To this end, for $j=0,\ldots,K-1$,
\[
\hat L^n_{t+1+j,K-j} = \sum_{i=0}^j \hat \D^n_{t+1+i,K-i} + \sum_{i=0}^j (\hat A^n_{t+1+i}
+ \hat Q^n_{t+i} - \hat Q^n_{t+1+i}) p_{K-i},
\]
which after a change of time indices renders, for $t \geq K$ and
$j=1,\ldots,K$,
\begin{equation}
\hat L^n_{t,j} = \sum_{i=1}^{K+1-j} \hat \D^n_{t+1-i,j+i-1} +
\sum_{i=1}^{K+1-j} (\hat A^n_{t+1-i} + \hat Q^n_{t-i} - \hat Q^n_{t-i+1}) p_{j+i-1}. \label{eq:Ljt}
\end{equation}
Summing both sides of~(\ref{eq:Ljt}) over $j=1,\ldots,K$ and using~(\ref{eq:hatMt}) results in
\begin{align*}
\sum_{j=1}^K \hat L^n_{t,j} &= \sum_{i=1}^K \sum_{j=i}^K \hat \D^n_{t+1-i,j} +
\sum_{i=1}^K  (\hat A^n_{t+1-i} + \hat Q^n_{t-i} - \hat Q^n_{t-i+1}) \tilde p_{i} \\
&= \hat \fM^n_t - \hat Q^n_t + \sum_{i=1}^K p_i \hat Q^n_{t-i}.
\end{align*}
The statement of the lemma follows from the preceding equality, (\ref{eq:hatY}) and~(\ref{eq:Y=Q-beta}).
\end{proof}

The following corollary establishes a lower and upper bound on
the value of $\hat Y^n_t$ in terms of the past values of $\{\hat
Y^n_t, t\in \rZ_+\}$ and the process $\{\hat \fM^n_t , t\in
\rZ_+\}$.

\begin{coro} \label{coro:Ybound} (i) For every $k \in \rZ_+$ and $t \geq k+K$
\begin{align*}
\hat Y^n_t \leq \sum_{i=0}^k (\hat \fM^n_{t-i})^+ + \sum_{i=k+1}^{k+K} \tilde p_{i-k}(\hat Y^n_{t-i}-\beta_n)^+.
\end{align*}
(ii) For $1 \leq i_1,\ldots,i_k \leq K$ let $p(k) = \prod_{j=1}^k p_{i_j}$ and $s(k) = \sum_{j=1}^k {i_j}$ with $s(0) = 0$.  Then for $t \geq s(k)$
\begin{equation*}
\hat Y^n_t  \geq p(k) \hat Y^n_{t-s(k)} - \sum_{j=0}^{k-1} (\beta - \hat \fM^n_{t-s(j)})^+.
\end{equation*}
\end{coro}

\begin{proof} The proofs of the two parts are by induction on $k$.

(i) The base of the induction ($k=0$) is due to Lemma~\ref{lemma:Yevo} and $\tilde p_i \geq p_i$ for $i=1,\ldots,K$.
Then the bound follows from the inductive assumption, Lemma~\ref{lemma:Yevo} and $\tilde p_i = p_i + \tilde p_{i+1} \leq 1$:
\begin{align*}
\hat Y^n_t &\leq \sum_{i=0}^k (\hat \fM^n_{t-i})^+ + (\hat Y^n_{t-k-1})^+ + \sum_{i=k+2}^{k+K} \tilde p_{i-k}(\hat Y^n_{t-i}-\beta_n)^+ \\
&\leq \sum_{i=0}^{k+1} (\hat \fM^n_{t-i})^+ + \sum_{i=k+2}^{k+1+K} \tilde p_{i-k-1}(\hat Y^n_{t-i}-\beta_n)^+.
\end{align*}

(ii)  By Lemma~\ref{lemma:Yevo} one has $\hat Y^n_t \geq \hat \fM^n_t + p_{i_1} (\hat Y^n_{t-i_1} - \beta)^+$ that implies
\begin{align}
\hat Y^n_t &\geq p_{i_1} \hat Y^n_{t-i_1} - (\beta - \hat \fM^n_t) \notag \\
&\ge p_{i_1} \hat Y^n_{t-i_1} - (\beta - \hat \fM^n_t)^+. \label{eq:coro3base}
\end{align}
The preceding inequality provides the base of the induction ($k=1$). Suppose
now that the statement of the corollary holds for some $k\ge 1$. Then
combining~(\ref{eq:coro3base}), the inductive assumption and $p_i \leq 1$
yields
\begin{align*}
\hat Y^n_t &\geq  p(k) \hat Y^n_{t-s(k)} - \sum_{j=0}^{k-1} (\beta - \hat \fM_{t-s(j)})^+ \\
&\geq p(k+1) \hat Y^n_{t-s(k+1)} - \sum_{j=0}^k (\beta - \hat \fM_{t-s(j)})^+,
\end{align*}
where the second inequality is also due to $s(k+1) = s(k) + i_{k+1}$.
\end{proof}

In the rest of the section we state the limiting counterparts of the results derived for $\{\hat Y^n_t, t\in\rZ_+\}$.
We start by introducing the limiting analogs of $\hat{\vct  Z}^n_t$ and $\hat \fM^n_t$. Define
\begin{equation}
\hat{\vct  Z}_t \triangleq \hat {\vct  \D}_t + \vct  p \hat A_t \label{eq:defZlimit}
\end{equation}
and
\begin{align}
\hat \fM_t &\triangleq \sum_{i=1}^K \sum_{j=i}^K \left(\hat \D_{t+1-i,j} + p_j \hat A_{t+1-i} \right) \nonumber \\ 
&=  \sum_{i=1}^K \sum_{j=i}^K \hat Z_{t+1-i,j}, \label{eq:MZLimit}
\end{align}
where $\hat Z_{t,i}$'s are the elements of $\hat {\vct Z}_t$. Since $\hat {\vct \D}_t$ and $\hat A_t$ are normal r.v.s by definition, $\hat {\vct Z}_t$ is normally distributed
as well and for all $t$ and $i$ we have
\begin{equation}\label{eq:ZDinftyLimit}
|\hat Z_{t,i}| \in\mathcal{M}_\infty.
\end{equation}

The properties of $\{\hat Y_t, t\in\rZ_+\}$ are summarized in the following lemma, including a limiting counterparts of Lemma~\ref{lemma:Yevo} and Corollary~\ref{coro:Ybound}.

\begin{lemma} \label{lemma:Ytsummary}
(i) The process $\{\hat Y_t, t\in\rZ_+\}$ satisfies for all $t \geq K$
\[
\hat Y_t = \hat \fM_t + \sum_{i=1}^K p_i (\hat Y_{t-i} - \beta)^+.
\]
(ii) For every $t \in \rZ_+$
\begin{equation*}
(-\hat Y_t)\in \mathcal{M}_\infty.
\end{equation*}
(iii) For every $k\in\rZ_+$ and $t \geq k+K$
\begin{equation*}
\hat Y_t \leq \sum_{i=0}^k (\hat \fM_{t-i})^+ + \sum_{i=k+1}^{k+K} \tilde p_{i-k}(\hat Y_{t-i}-\beta)^+.
\end{equation*}
(iv) For $1\le i_1,\ldots,i_k\le K$ let $p(k) = \prod_{j=1}^k p_{i_j}$ and $s(k)= \sum_{j=1}^k i_j$ with $s(0)=0$.  Then, for $t \geq s(k)$
\begin{equation*}
\hat Y_t  \geq p(k) \hat Y_{t-s(k)} - \sum_{j=0}^{k-1} (\beta - \hat \fM_{t-s(j)})^+.
\end{equation*}
\end{lemma}

\begin{proof}
(i) The proof is analogous to the proof of Lemma~\ref{lemma:Yevo}. Part (ii) follows from $\hat Y_t \geq \hat V_t$ (i)
and the fact that $\hat A_t$ and $\hat {\vct \D}_t$ have normal distributions. The proofs of (iii) and (iv) are analogous to the proof of Corollary~\ref{coro:Ybound}.
\end{proof}

\subsection{Quadratic Lyapunov function}\label{section:Lyapunov2}
Here  we introduce a quadratic Lyapunov function and prove some of
its properties. To this end, we define $K$ vectors
$\vct\alpha_1,\ldots,\vct\alpha_K$, where elements of the vector
$\vct\alpha_k=(\alpha_{k,1},\ldots,\alpha_{k,K})$ are defined by
\begin{equation}\label{eq:alphaDef}
\alpha_{k,j}=(j-k)^+;
\end{equation}
let $\vct\alpha = (\vct\alpha_1,\ldots,\vct\alpha_K) \in \rR^{K^2}$.
Let a function $\Psi_\theta(\vct  y, {\vct  z}): \rR^{K+K^2} \to
\rR$ be defined by
\begin{equation}
\Psi_\theta(\vct  y, {\vct  z}) \triangleq \left(\tilde{\vct  p} \cdot {\vct  y} \,
 + \, {\vct\alpha} \cdot  {\vct  z} \right)^\theta \label{eq:PsiDef}
\end{equation}
and a set ${\cal R}_x$ by
\begin{equation}
{\cal R}_x \triangleq \{{\vct y} \in \rR^K: y_i < x \text{ for some } i \}. \label{eq:Rx}
\end{equation}
The case $\theta=2$ is of particular importance since it corresponds to a quadratic Lyapunov
function (see Appendix~B for the definition) as established below. Finally, we introduce $\bar {\vct  Y}^n_t \triangleq (\hat
Y^n_t, \ldots, \hat Y^n_{t-K+1})$ and $\bar {\vct  Z}^n_t \triangleq
(\hat{\vct  Z}^n_t, \ldots, \hat{\vct  Z}^n_{t-K+1})$; the ``bar"
symbol in $\bar {\vct  Y}^n_t$ and $\bar {\vct  Z}^n_t$ indicates
that elements of these vectors refer to different time indices.

\begin{prop}\label{prop:2LyapunovBound} There exist $\delta>0$, $\psi<\infty$ and $n_0$ such that for all $n\geq n_0$
\begin{equation*}
\expect \left[\Psi_2(\bar {\vct  Y}^n_t, \bar {\vct  Z}^n_t) \, \ind{\bar {\vct  Y}^n_{t-1} \notin{\cal R}_{\beta_n}} - \Psi_2(\bar {\vct  Y}^n_{t-1}, \bar {\vct  Z}^n_{t-1})\, \big| \,\bar {\vct  Y}^n_{t-1}, \bar {\vct  Z}^n_{t-1} \right] \leq -\delta \Psi_1(\bar {\vct  Y}^n_{t-1}, \bar {\vct  Z}^n_{t-1}) + \psi. 
\end{equation*}
\end{prop}

\begin{proof}
On the event $\{\bar {\vct  Y}^n_{t-1} \notin  {\cal R}_{\beta_n}\}$ Lemma~\ref{lemma:Yevo} renders in a vector form $\hat Y^n_t = \hat \fM^n_t - \beta_n + {\vct  p} \cdot \bar {\vct  Y}^n_{t-1}$, and, since $p_i + \tilde p_{i+1} = \tilde p_i$ by definition, it implies $\tilde {\vct  p} \cdot \bar {\vct  Y}^n_t = \hat \fM^n_t - \beta_n + \tilde {\vct  p} \cdot \bar {\vct  Y}^n_{t-1}$. Thus, the linear combination of $\bar {\vct Y}^n_t$ and $\bar {\vct Z}^n_t$ that appears in the definition of $\Psi_\theta$ can be expressed as
\begin{align}
\tilde {\vct  p} \cdot \bar {\vct  Y}^n_t +  {\vct\alpha} \cdot \bar {\vct  Z}^n_t
&= \hat \fM^n_t - \beta_n + \tilde {\vct  p} \cdot \bar {\vct  Y}^n_{t-1} + {\vct\alpha} \cdot \bar {\vct  Z}^n_t \nonumber \\
&= \tilde {\vct  p} \cdot \bar {\vct  Y}^n_{t-1} +
{\vct\alpha} \cdot \bar {\vct  Z}^n_{t-1} - \beta_n + {\vct K} \cdot  \hat {\vct Z}^n_t, \label{eq:2expDecomp}
\end{align}
where the second equality follows from~(\ref{eq:alphaDef}) and~(\ref{eq:MZ}). Then, based on~(\ref{eq:2expDecomp}), we obtain
\begin{align}
\expect &\left[\Psi_2(\bar {\vct  Y}^n_t, \bar {\vct Z}^n_t)
\ind{\bar {\vct Y}^n_{t-1} \notin{\cal R}_{\beta_n}}\, \big|\, \bar {\vct Y}^n_{t-1}, \bar {\vct  Z}^n_{t-1} \right] -  \Psi_2(\bar {\vct  Y}^n_{t-1}, \bar {\vct Z}^n_{t-1}) \notag \\
&\phantom{xxx}\leq 2 \Psi_1(\bar {\vct  Y}^n_{t-1}, \bar {\vct  Z}^n_{t-1}) \, \expect \left[-\beta_n + {\vct K} \cdot \hat {\vct Z}^n_t \, \big|\, \bar {\vct Y}^n_{t-1}, \bar {\vct  Z}^n_{t-1} \right] + \expect \left[\left(-\beta_n + {\vct K} \cdot \hat {\vct Z}^n_t \right)^2\, \!\!\big|\, \bar {\vct Y}^n_{t-1}, \bar {\vct  Z}^n_{t-1} \right]. \label{eq:2EPhiPhi1n}
\end{align}

Now, by~(\ref{eq:defZ}) the sum in~(\ref{eq:2EPhiPhi1n})
can be expressed in terms of $\hat A^n_t$ and $\hat {\vct \D}^n_t$ rendering
\begin{align}
\expect \left[{\vct K} \cdot \hat {\vct Z}^n_t  \, \big|\,  \bar {\vct Y}^n_{t-1}, \bar {\vct Z}^n_{t-1} \right]
&= \expect \left[\hat A^n_t /\mu + {\vct K} \cdot \hat {\vct \D}^n_t \, \big| \, \bar {\vct Y}^n_{t-1}, \bar {\vct Z}^n_{t-1} \right] \notag\\
&= \expect \left[\hat A^n_t /\mu \, \big| \, \bar {\vct Y}^n_{t-1}, \bar {\vct Z}^n_{t-1} \right] \nonumber\\
&\leq \sup_{a \geq 0} \expect \left[\hat A^n_t/\mu \,\big|\, a^n_{t-1} = a \right], \label{eq:AandD}
\end{align}
where the last inequality is due to the fact that $A^n_t$ is
conditionally independent of $\hat {\vct \D}^n_t$ and $(\bar {\vct
Y}^n_{t-1}, \bar {\vct Z}^n_{t-1})$  given $a_{t-1}$ (the arrival
process is renewal). The second expectation on the right-hand side
of~(\ref{eq:2EPhiPhi1n}) can be upper bounded by utilizing the the
same fact in addition to observation that~$\hat {\vct \D}^n_t$ is
conditionally independent of $\hat A^n_t$ and $(\bar {\vct
Y}^n_{t-1}, \bar {\vct Z}^n_{t-1})$ given $\hat J^n_t$ -- see
Corollary~\ref{coro:MCscaled}. These two facts yield
\begin{align}
\expect \left[\left(-\beta_n + {\vct K} \cdot \hat {\vct Z}^n_t
\right)^2\, \big|\, \bar {\vct Y}^n_{t-1}, \bar {\vct  Z}^n_{t-1}
\right] &= \expect \left[\left(\hat A^n_t/\mu -\beta_n \right)^2 +
\left( {\vct K} \cdot \hat {\vct \D}^n_t \right)^2\, \big|\,
\bar {\vct Y}^n_{t-1}, \bar {\vct  Z}^n_{t-1} \right] \notag\\
&\hspace{-2in}\leq \sup_{a \geq 0}
 \expect\left[\left(\hat A^n_t/\mu -\beta_n\right)^2 \, \bigg|\, a^n_{t-1}=a \right]
 + \max_{0 \leq i \leq n} \expect\left[\left( {\vct K}\cdot \hat {\vct \D}^n_t \right)^2 \big|\, \| {\vct \D}^n_t\| = i \right] \nonumber\\
&\hspace{-2in}\leq \sup_{a \geq 0}
\expect\left[\left(\hat A^n_t/\mu -\beta_n\right)^2 \, \bigg|\, a^n_{t-1}=a \right] + \sigma^2_s, \label{eq:AandD2}
\end{align}
where the last inequality follows from Lemma~\ref{lemma:Dbound} and
$\|{\vct \D}^n_t\| \leq n$.  The limit (as $n \to \infty$) of the right-hand
side of the preceding inequality remains bounded due to the
assumption~(\ref{eq:arrival-A2}) on the arrival process
(Section~\ref{sec:GGN}) and the fact that service times are bounded
($S \leq K$).

Combining~(\ref{eq:2EPhiPhi1n}) with~(\ref{eq:AandD}),
(\ref{eq:arrival-A1}) and~(\ref{eq:AandD2})  yields the statement of
the theorem.
\end{proof}

\begin{lemma} \label{lemma:inRbound} The following inequality holds
\[
\limsup_{n\rightarrow\infty}
\expect_{\pi_n}  \left[ \Psi_2(\bar {\vct Y}^n_t, \bar {\vct  Z}^n_t) \,\ind{\bar {\vct  Y}^n_{t-1} \in {\cal
R}_{\beta_n}} \right] < \infty.
\]
\end{lemma}

In the proof of the lemma the following number-theoretic fact will be utilized. For completeness we provide its proof.

\begin{lemma}  \label{lemma:number}
Let $p$ and $q$ be two relatively prime numbers. For any $K \in
\mathbb{N}$  there exists $k \in \mathbb{N}$ such that any $l \in
\{k+1,\ldots,k+K\}$ can be represented as $l=i_l p + j_l q$ for some
$i_l, j_l \in \mathbb{N}$.
\end{lemma}

\begin{proof}
Since $p$ and $q$ are relatively prime then any $m \in
\{1,2,\ldots,K\}$ can be  represented as $m = i'_n p + j'_n q$ for
some possibly negative integers $i'_n$ and $j'_n$, see
e.g.~\cite[p.~104]{LPV03}. Let $t = \max_m \{i'_n, j'_n \} + 1$ and
$k=tp+tq$. Then every $l\in\{k+1,\ldots,k+K\}$ is given by $l=i_l p
+ j_l q$, where $i_l = (t+ i'_{l-k})$ and $j_l = (t+ j'_{l-k})$.
\end{proof}

\begin{proof}[Proof of Lemma~\ref{lemma:inRbound}]
Let ${\cal R}_{x}^k = \{\vct  y \in \rR^K:\,  y_1 \geq x, \ldots,
y_{k-1} \geq x, y_k < x\}$, $1\leq k \leq K$. It is sufficient
to prove the  statement of the lemma with ${\cal R}_{\beta_n}$
replaced with ${\cal R}^k_{\beta_n}$ for an arbitrary
$k\in\{1,\ldots,K\}$ since ${\cal R}_{\beta_n} = \cup_k {\cal
R}^k_{\beta_n}$. The proof is based on demonstrating the following
bound for some positive integer $m$ and constants $\{c_i, i=0,\ldots,m+k+K\}$, $\{d_i, i=0,\ldots,m+k+K\}$
such that for all $n$
\begin{equation}
\Psi_2(\bar {\vct  Y}^n_t, \bar {\vct  Z}^n_t) \,\ind{\bar {\vct  Y}^n_{t-1} \in {\cal R}^k_{\beta_n}}
\leq \left(\sum_{i=0}^{m+k+K} (c_i + d_i |\hat \fM^n_{t-i}|) \right)^2. \label{eq:E<inftyS1}
\end{equation}
Then the statement of the lemma follows from the definition of $\hat
\fM^n_t$,  Lemma~\ref{lemma:Dbound} and~(\ref{eq:arrival-A2}) applied in the unconditioned case; thus, we focus on demonstrating~(\ref{eq:E<inftyS1}).

On event $\{\bar {\vct  Y}^n_{t-1} \in {\cal R}_{\beta_n}^k \}$ applying Lemma~\ref{lemma:Yevo} to $\hat Y^n_t$ yields
\begin{align}
\hat Y^n_t &= \hat \fM^n_t + \sum_{i=1}^{k-1} p_i (\hat Y^n_{t-i} - \beta_n)  + \sum_{i=k+1}^K  p_i (\hat Y^n_{t-i} - \beta_n)^+ \nonumber \\
&= \hat \fM^n_t + \sum_{i=1}^{k-1} g_i (\hat \fM^n_{t-i} -\beta_n) + \sum_{i=k+1}^{K+k-1} h_i (\hat Y^n_{t-i} - \beta_n)^+, \label{eq:Y<inftyeq1}
\end{align}
where the constants $g_i$'s and $h_i$'s can be computed in a
recursive  fashion: $g_0=1$, $g_i = \sum_{j=0}^{i-1} g_j p_{i-j}$,
for $i=1,\ldots,k-1$, and $h_i  = \sum_{j=(i-K)^+}^{k-1} g_j
p_{i-j}$, for $i=k+1,\ldots,K+k-1$. Hence, based
on~(\ref{eq:Y<inftyeq1}), there exist finite $g$ and $h$ such that
\begin{equation}
\tilde {\vct  p} \cdot \bar {\vct  Y}^n_t \leq
g \sum_{i=0}^{k-1} (\hat \fM^n_{t-i})^+ + h \sum_{i=k+1}^{K+k-1} (\hat Y^n_{t-i})^+. \label{eq:Y<inftyeq5}
\end{equation}
Next, on the event of interest, $\{\bar {\vct Y}^n_{t-1} \in {\cal
R}^k_{\beta_n}\}$, we upper bound the second sum
in~(\ref{eq:Y<inftyeq5}) in two steps: (i) bound values of $\{\hat
Y^n_t, t \in \rZ_+\}$ on a time interval of lenght $K$ prior to time
$(t-k)$ based on $\{\hat Y^n_{t-k} < \beta_n\}$, and (ii) obtain a
desired bound based on (i). First, consider arbitrary $i_1, i_2 \leq
K$ such that $p_{i_1} p_{i_2}>0$; such a pair of indices exists
since $\sigma_s>0$ (see Section~\ref{sec:GGN}). By
Lemma~\ref{lemma:number} there exists a sufficiently large $m$ such
that every element of $\{m+1, m+2, \ldots, m+K\}$ can be represented
as $r_1 i_1 + r_2 i_2$ for some nonnegative integers $r_1$ and
$r_2$. Invoking the second part of Corollary~\ref{coro:Ybound} and
$\{\hat Y^n_{t-k} < \beta_n\}$  yields the existence of finite $r$,
$q$ and $m \geq K$ such that
\[
(\hat Y^n_{t-k-i})^+ \leq r + q \sum_{j=k}^{m+k+K} |\hat \fM^n_{t-j}|,
\]
for all $i \in \{m+1,\ldots,m+K\}$; we also used $|x+y| \leq
|x|+|y|$ and the fact the  elements of the sum are nonnegative. The
preceding inequality and the first part of
Corollary~\ref{coro:Ybound} assure the existence of finite $r'$ and
$q'$ such that
\begin{equation}
h \sum_{i=k+1}^{K+k-1} (\hat Y^n_{t-i})^+ \leq r' + q' \sum_{i=k}^{m+k+K} |\hat \fM^n_{t-i}| \label{eq:Y<inftyeq2}
\end{equation}
since each summand on the left-hand side is upper bounded by an
expression that appears on  the right-hand side with $r'$ and $q'$
replaced by some other finite constants.

Next, combining~(\ref{eq:Y<inftyeq5}) and~(\ref{eq:Y<inftyeq2}) provides
a bound on $\tilde {\vct  p} \cdot \bar {\vct  Y}^n_t$ in terms of $\hat \fM^n_t,\ldots, \hat \fM^n_{t-m-k-K}$:
\begin{align*}
\tilde {\vct  p} \cdot \bar {\vct  Y}^n_t
&\leq g \sum_{i=0}^{k-1} (\hat \fM^n_{t-i})^+ + r' + q' \sum_{i=k}^{m+k+K} |\hat \fM^n_{t-i}| \\
&\leq r' + g' \sum_{i=0}^{m+k+K} |\hat \fM^n_{t-i}|,
\end{align*}
where $g'$ is finite. Finally, from (\ref{eq:MZ}) we have that the absolute value of
${\vct\alpha} \cdot \bar {\vct  Z}^n_t$ is upper bounded by a linear combination of $|\hat \fM^n_{t-i}|$'s. Then, (\ref{eq:E<inftyS1}) follows from the preceding bound. This completes the proof.
\end{proof}

\begin{lemma} \label{lemma:<0} The following inequality holds
\[
\limsup_{n\rightarrow\infty}
\expect_{\pi_n}
\left[- \Psi_1(\bar {\vct Y}^n_t, \bar {\vct  Z}^n_t) \,\ind{\Psi_1(\bar {\vct  Y}^n_t, \bar {\vct Z}^n_t)<0} \right] < \infty.
\]
\end{lemma}

\begin{proof}
Lemma~\ref{lemma:Yevo} renders $\hat Y^n_t \geq \hat V^n_t$ that leads to
\[
\Psi_1(\bar {\vct Y}^n_t, \bar {\vct  Z}^n_t)
\geq  \sum_{i=1}^K \tilde p_i \hat \fM^n_{t+1-i} + {\mb \alpha} \cdot \bar {\vct  Z}^n_t.
\]
The statement follows from the preceding relationship, (\ref{eq:hatMt}), Cauchy-Schwarz inequality, (\ref{eq:arrival-A2}) and Lemma~\ref{lemma:Dbound}.
\end{proof}

\subsection{Geometric Lyapunov function}\label{section:LyapunovFunction}

In this section we introduce a family of Lyapunov functions
parameterized by some $\theta>0$ and prove some of its properties.
Given a parameter $\theta>0$, consider a function $\Phi_\theta (\vct
y, {\vct  z}): \rR^{K+K^2} \to \rR_+$ defined by
\begin{equation}
\Phi_\theta(\vct  y,  {\vct  z}) \triangleq \exp \left\{ \theta \tilde{\vct  p} \cdot {\vct  y} \,
 + \, \theta {\vct\alpha} \cdot  {\vct  z}\right\}. \label{eq:PhiDef}
\end{equation}
We first consider $\Phi_\theta$ as a function of the limiting pair $(\bar
{\vct Y}_t, \bar {\vct Z}_t)$,  where $\bar {\vct  Y}_t \triangleq
(\hat Y_t, \ldots, \hat Y_{t-K+1})$ and $\bar {\vct  Z}_t \triangleq
(\hat{\vct  Z}_t, \ldots, \hat{\vct  Z}_{t-K+1})$. The next
proposition establishes a negative drift of the Lyapunov function
$\Phi_\theta$ under an assumption $\theta<\theta^*/\mu$. Moreover,
$\theta^*/\mu$ is the critical exponent under which $\Phi_\theta$ is
a geometric Lyapunov function (see Appendix~B for the definition). Recall the definition of ${\cal R}_x$ from~(\ref{eq:Rx}).

\begin{prop}\label{prop:LyapunovBoundLimit} For every $\theta < \theta^*/\mu$ there exists
$\delta=\delta_\theta>0$ such that
\begin{equation}
\expect \left[\Phi_\theta(\bar {\vct  Y}_t, \bar {\vct  Z}_t) \,
\ind{\bar {\vct  Y}_{t-1}\notin{\cal R}_{\beta}} \, \big| \bar {\vct  Y}_{t-1},
\bar {\vct  Z}_{t-1} \right] \leq (1-\delta) \, \Phi_\theta (\bar {\vct  Y}_{t-1}, \bar {\vct  Z}_{t-1})
\label{eq:PhidecreaseLimit}
\end{equation}
and for every $\theta>\theta^*/\mu$ there exists $\delta=\delta_\theta>0$ such that
\begin{equation}
\expect \left[ \Phi_\theta(\bar {\vct  Y}_t, \bar {\vct  Z}_t)
\, \big|\bar {\vct  Y}_{t-1}, \bar {\vct  Z}_{t-1} \right]
\geq (1+\delta) \, \Phi_\theta (\bar {\vct  Y}_{t-1}, \bar {\vct  Z}_{t-1}).   \label{eq:PhiincreaseLimit}
\end{equation}
\end{prop}

\begin{proof}
The analysis is similar to the one of Proposition~\ref{prop:2LyapunovBound}. From Lemma~\ref{lemma:Ytsummary}(i) we have that on the event $\{\bar {\vct  Y}_{t-1}\notin{\cal R}_{\beta}\}$
\begin{align}
\tilde {\vct  p} \cdot \bar {\vct  Y}_t +  {\vct\alpha} \cdot \bar {\vct  Z}_t
&= \hat \fM_t - \beta + \tilde {\vct  p} \cdot \bar {\vct  Y}_{t-1} + {\vct\alpha} \cdot \bar {\vct  Z}_t \nonumber \\
&= \tilde {\vct  p} \cdot \bar {\vct  Y}_{t-1} +
{\vct\alpha} \cdot \bar {\vct  Z}_{t-1} - \beta + {\vct K} \cdot \hat {\vct Z}_t. \label{eq:2expDecompLimit}
\end{align}
This results in
\begin{align}
\expect &\left[\Phi_\theta(\bar {\vct  Y}_t, \bar {\vct Z}_t)
\ind{\bar {\vct  Y}_{t-1}\notin{\cal R}_{\beta}}\,\big|\, \bar {\vct Y}_{t-1}, \bar {\vct  Z}_{t-1} \right] \notag\\
&\phantom{xxx}= e^{-\theta \beta}\,\expect \left[\exp \left\{\theta {\vct K} \cdot \hat {\vct Z}_t \right\}
\ind{\bar {\vct  Y}_{t-1}\notin{\cal R}_{\beta}}\,
\big|\,  \bar {\vct Y}_{t-1}, \bar {\vct Z}_{t-1} \right]  \,
\Phi_\theta (\bar {\vct  Y}_{t-1}, \bar {\vct Z}_{t-1}) \nonumber \\
&\phantom{xxx}\le e^{-\theta \beta}\,\expect \left[\exp \left\{\theta {\vct K} \cdot \hat {\vct Z}_t \right\}\,
\big|\,  \bar {\vct Y}_{t-1}, \bar {\vct Z}_{t-1} \right]  \,
\Phi_\theta (\bar {\vct  Y}_{t-1}, \bar {\vct Z}_{t-1}) \nonumber \\
&\phantom{xxx}= e^{-\theta \beta}\, \expect \left[ e^{\theta \hat A_t /\mu} \right] \,
\expect \left[\exp \left\{\theta {\vct K} \cdot \hat {\vct \D}_t \right\} \right]  \,
\Phi_\theta (\bar {\vct  Y}_{t-1}, \bar {\vct Z}_{t-1}), \label{eq:EPhiPhi1}
\end{align}
where the last equality follows from the definition of $\hat {\vct
Z}_t$,  mutual independence of $\hat A_t$ and $\hat {\vct \D}_t$ as
well as their independence of $(\bar {\vct  Y}_{t-1}, \bar {\vct
Z}_{t-1})$. By definition, r.v. $\hat A_t$ is normally
distributed with zero mean and variance $\mu c^2_a$ and, hence,
\begin{equation}
\expect \left[e^{\theta \hat A_t /\mu}\right] = e^{\theta^2 c_a^2/(2\mu)}. \label{eq:MAconv}
\end{equation}
On the other hand, $\hat{\vct \D}_t$ is normal with the covariance
matrix $\mu\Sigma = \mu (\text{diag}({\vct p}) - {\vct p}^T{\vct p})$ (see~(\ref{eq:Sigma})) where $\text{diag}(\vct p)$ is the diagonal matrix defined by $\vct p$. Thus 
\begin{align*}
\E \left({\vct K} \cdot \hat {\vct \D}_t\right)^2 &=\mu {\vct K}^T(\text{diag}({\vct p}) - {\vct p}^T{\vct p}) {\vct K} \\
&=\mu\sum_{j=1}^K j^2p_j - \mu \left({\vct K} \cdot {\vct p} \right)^2 =\mu\sigma_s^2,
\end{align*}
which in turn implies
\[
\expect \left[\exp \left\{\theta {\vct K} \cdot  \hat {\vct \D}_t \right\}\right] = e^{\theta^2 c_s^2/(2\mu)}.
\]
Inequality~(\ref{eq:EPhiPhi1}), (\ref{eq:MAconv}) and the preceding equality result in
\begin{equation*}
\expect \left[ \Phi_\theta(\bar {\vct  Y}_t, \bar {\vct Z}_t)
\ind{\bar {\vct  Y}_{t-1}\notin{\cal R}_{\beta}}\, \big|\, \bar {\vct Y}_{t-1}, \bar {\vct  Z}_{t-1} \right] \leq e^{-\theta \beta} e^{\theta^2 c_a^2/(2\mu)} e^{\theta^2 c_s^2/(2\mu)}
\Phi_\theta (\bar {\vct  Y}_{t-1}, \bar {\vct Z}_{t-1}),
\end{equation*}
and (\ref{eq:PhidecreaseLimit}) then follows provided that
\[
-\theta\beta+ \theta^2 c_a^2/(2\mu) + \theta^2 c_s^2/(2\mu)<0,
\]
or equivalently $\theta<\theta^*/\mu$. Therefore, the first part of the proposition is established.

The proof of (\ref{eq:PhiincreaseLimit}) is very similar.  We
observe from Lemma~\ref{lemma:Ytsummary}(i) that
$\hat Y_t \ge \hat V_t - \beta + {\vct  p} \cdot \bar {\vct  Y}_{t-1}$, regardless of
whether $\bar{\vct Y}_{t-1}\in \mathcal{R}_\beta$ or $\bar{\vct
Y}_{t-1} \not\in \mathcal{R}_\beta$. Repeating the analysis for the
previous case ($\theta<\theta^*/\mu$), we obtain
\begin{align*}
\expect \left[\Phi_\theta(\bar {\vct  Y}_t, \bar {\vct Z}_t)
 |\, \bar {\vct Y}_{t-1}, \bar {\vct  Z}_{t-1} \right] &\geq e^{-\theta \beta}\,\expect \left[\exp \left\{\theta {\vct K} \cdot \hat {\vct Z}_t \right\} \, \big|\,  \bar {\vct  Y}_{t-1}, \bar {\vct  Z}_{t-1} \right]  \,
\Phi_\theta (\bar {\vct  Y}_{t-1}, \bar {\vct Z}_{t-1}) \notag \\
&= e^{-\theta \beta + \theta^2 c_a^2/(2\mu) + \theta^2 c_s^2/(2\mu)} \,
\Phi_\theta (\bar {\vct  Y}_{t-1}, \bar {\vct Z}_{t-1});
\end{align*}
thus, (\ref{eq:PhiincreaseLimit}) holds provided that $\theta>\theta^*/\mu$.
This concludes the proof of the proposition.
\end{proof}

The following analogue of Proposition~\ref{prop:LyapunovBoundLimit} is needed to establish our third main result, Theorem~\ref{theorem:MainResultExpErg}.

\begin{prop}\label{prop:LyapunovBoundn}
There exist $\theta>0,0<\delta<1$ and $n_0$, such that for all $n\ge n_0$,
\begin{equation}
\expect \left[\Phi_\theta(\bar {\vct  Y}^n_t, \bar {\vct  Z}^n_t) \,
\ind{\bar {\vct  Y}^n_{t-1}\notin{\cal R}_{\beta_n}} \, \big| \bar {\vct  Y}^n_{t-1},
\bar {\vct  Z}^n_{t-1} \right] \leq (1-\delta) \, \Phi_\theta (\bar {\vct  Y}^n_{t-1}, \bar {\vct  Z}^n_{t-1}).
\label{eq:Phidecreasen}
\end{equation}
\end{prop}

\begin{proof}
Repeating the first steps of the proof of Proposition~\ref{prop:LyapunovBoundLimit}, we obtain
\begin{align}
\expect\Big[ &\Phi_\theta(\bar {\mb  Y}^n_t, \bar {\mb Z}^n_t)
\ind{\bar {\mb  Y}^n_{t-1}\notin{\cal R}_{\beta_n}}\, \big|\, \bar {\mb Y}^n_{t-1}, \bar {\mb  Z}^n_{t-1} \Big] \notag\\
&=e^{-\theta \beta_n}\,\expect \left[\exp \left\{\theta {\vct K} \cdot \hat {\vct Z}^n_t \right\}
\ind{\bar {\mb  Y}^n_{t-1}\notin{\cal R}_{\beta_n}} \, \big|\,  \bar {\mb  Y}^n_{t-1}, \bar {\mb  Z}^n_{t-1} \right]  \,
\Phi_\theta (\bar {\mb  Y}^n_{t-1}, \bar {\mb Z}^n_{t-1}) \notag \\
&\le e^{-\theta \beta_n}\,\expect \left[\exp \left\{\theta {\vct K} \cdot \hat {\vct Z}^n_t \right\} \,
\big|\,  \bar {\mb  Y}^n_{t-1}, \bar {\mb  Z}^n_{t-1} \right]  \,
\Phi_\theta (\bar {\mb  Y}^n_{t-1}, \bar {\mb Z}^n_{t-1}). \label{eq:3EPhiPhi1}
\end{align}
Now, the expectation in~(\ref{eq:3EPhiPhi1}) can be expressed in terms of $\hat A^n_t$ and $\hat {\vct \D}^n_t$ the following way
\begin{equation}
\expect  \left[\exp \left\{\theta {\vct K}\cdot \hat {\vct Z}^n_t \right\} \,
\big|\, \bar {\mb  Y}^n_{t-1}, \bar {\mb  Z}^n_{t-1} \right]
= \expect  \left[\exp \left\{\theta \hat A^n_t /\mu +\theta {\vct K}\cdot \hat {\vct \D}^n_t\right\}
\big| \bar {\mb  Y}^n_{t-1}, \bar {\mb  Z}^n_{t-1}\right]. \label{eq:3AandD}
\end{equation}
The right-hand side can be upper bounded by utilizing: (i) $\hat A^n_t$ is conditionally independent of $\hat {\vct \D}^n_t$ and $(\bar {\mb  Y}^n_{t-1}, \bar {\mb  Z}^n_{t-1})$ given $a_{t-1}$ since the arrival process is renewal, and (ii) $\hat {\vct \D}^n_t$ is conditionally independent of $\hat A^n_t$ and $(\bar {\mb  Y}^n_{t-1}, \bar {\mb  Z}^n_{t-1})$ given $\hat \D^n_t$ (see Corollary~\ref{coro:MCscaled}). These two facts and~(\ref{eq:3AandD}) yield
\begin{align}
\expect \left[\exp \left\{\theta {\vct K}\cdot \hat {\vct Z}^n_t \right\} \big| \bar {\mb  Y}^n_{t-1}, \bar {\mb  Z}^n_{t-1} \right]
&\leq \sup_{a} \expect\left[e^{\theta \hat A^n_t/\mu} \big| a^n_{t-1}=a \right] \,
\max_{0\leq l \leq n} \expect \left[\exp\left\{\theta {\vct K}\cdot \hat {\vct \D}^n_t \right\} \big|
\|{\vct \D}^n_t\|=l \right] \nonumber \\
&\leq \sup_{a} \expect\left[e^{\theta \hat A^n_t/\mu} \big| a^n_{t-1}=a \right]  \,
\left(\expect e^{\theta(S-1/\mu)/\sqrt{n}} \right)^n, \label{eq:24jul003}
\end{align}
where the second inequality follows from Lemma~\ref{lemma:Dbound}.
Using a second order Taylor expansion in conjunction with the observation that service times are bounded $(S \leq K)$, we have
\begin{equation}
\lim_{n\rightarrow\infty}\left(\expect e^{\theta (S - 1/\mu)/\sqrt{n}}\right)^n=
\lim_{n\rightarrow\infty}\left(1+ \frac{\theta^2 \sigma^2_s}{2n}+o(1/n)\right)^n
=e^{\theta^2 \sigma^2_s/2}<\infty. \label{eq:24jul001}
\end{equation}

Finally, from~(\ref{eq:24jul001}), (\ref{eq:24jul003}) and (\ref{eq:3EPhiPhi1})
we conclude that~(\ref{eq:Phidecreasen}) holds for $\theta$ that satisfy
\[
(c + \sigma^2_s/2) \theta^2 - \beta_n \theta < 0,
\]
that, in light of $\beta_n\rightarrow\beta$, is satisfied for sufficiently small $\theta$.
This establishes~(\ref{eq:Phidecreasen}) and concludes the proof.
\end{proof}

The analogue of Lemma~\ref{lemma:inRbound} for the geometric function applied to the $n$th and limiting processes is stated next. The proof is very similar
to that of Lemma~\ref{lemma:inRbound}, except that the fact that
$\hat A_t$ and $\hat{\vct \D}_t$ are normally distributed is utilized in the limiting case.

\begin{lemma} \label{lemma:inRboundLimit}
For every $\theta>0$ the following assertion holds
\[
\expect_{\pi_*}  \left[ \Phi_\theta(\bar {\vct Y}_t, \bar {\vct  Z}_t) \, \ind{\bar {\vct  Y}_{t-1} \in {\cal
R}_{\beta}} \right] < \infty.
\]
Moreover, under the assumption (\ref{eq:arrival-Aexp}), for every $\theta>0$
\[
\limsup_{n\rightarrow\infty}
\expect_{\pi_n}  \left[ \Phi_\theta(\bar {\vct Y}^n_t, \bar {\vct  Z}^n_t) \, \ind{\bar {\vct  Y}^n_{t-1} \in {\cal R}_{\beta_n}} \right] < \infty.
\]
\end{lemma}

\begin{proof}
As in the proof of Lemma~\ref{lemma:inRbound}, it is sufficient to
prove the  statement of the lemma with ${\cal R}_{\beta}$ replaced
with ${\cal R}^k_{\beta_n}$ for an arbitrary $k\in\{1,\ldots,K\}$.
Repeating the steps of the proof of Lemma~\ref{lemma:inRbound}
yields the existence of some positive integer $m$ and constants
$c_i$'s, $d_i$'s such that
\begin{equation*}
\Phi_\theta(\bar {\vct  Y}_t, \bar {\vct  Z}_t) \,\ind{\bar {\vct  Y}_{t-1} \in {\cal R}^k_{\beta}}
\leq \exp \left\{\sum_{i=0}^{m+k+K} (c_i + d_i |\hat \fM_{t-i}|) \right\}.
\end{equation*}
Then the statement of the lemma follows from the definition of $\hat
\fM_t$, the  Gaussian distribution of its components and
Proposition~\ref{prop:DinftyLinComb} in the appendix.

The proof of the second inequality is very similar and  uses~(\ref{eq:arrival-Aexp})
and Lemma~\ref{lemma:Dbound}.
\end{proof}

\section{Proof of Theorem~\ref{theorem:MainResultTightness}}\label{sec:Thm1}

The convergence in the statement of the theorem is established by proving the
tightness of all relevant random variables. Recall that a sequence of r.v.s $\{X_n, n\geq 1\}$ is tight~\cite[p.~87]{Dur05} if for all $\epsilon>0$ there is an $x_\epsilon$ so that
\[
\sup_{n \to \infty} \Pr\left[X_n \not\in (-x_\epsilon, x_\epsilon]\right] \leq \epsilon.
\]

\begin{prop}\label{prop:YthetafinitePrelim}
The sequence $\{\hat Y^n_t, n \ge 1\}$ is tight with respect to the
sequence of probability measures $\{\pi_n, n\geq 1\}$.
\end{prop}

\begin{proof}
Theorem~\ref{theorem:BoundsStationaryPolynomial} can be used to bound the sequence $\{\hat Y^n_t, n \ge 1\}$ away from $+\infty$. In order to obtain uniform boundedness away from $-\infty$ we utilize the fact that the negative part of $\hat Y^n_t$ can be upper bounded by $(\hat V^n_t)^-$ according to Lemma~\ref{lemma:Ytsummary}(i).

From Proposition~\ref{prop:2LyapunovBound}, Lemma~\ref{lemma:inRbound},  Lemma~\ref{lemma:<0} and
Theorem~\ref{theorem:BoundsStationaryPolynomial} it follows that
\begin{equation*}
\limsup_{n\rightarrow\infty} \expect_{\pi_n} \Psi_1 (\bar {\vct  Y}^n_t, \bar {\vct  Z}^n_t) < \infty.
\end{equation*}
Applying (\ref{eq:MZ}), (\ref{eq:arrival-A2})  and Lemma~\ref{lemma:Dbound}
we obtain from~(\ref{eq:PsiDef}) in the case $\theta=1$ that
\begin{equation}
\limsup_{n \to \infty} \expect_{\pi_n} \tilde{\vct  p}\cdot \bar {\vct  Y}^n_t < \infty. \label{eq:PR100e1}
\end{equation}
Next, Lemma~\ref{lemma:Yevo} implies $\hat Y^n_t \geq \hat V^n_t$ leading to
\[
(\hat Y^n_t)^- \leq |\hat V^n_t | = \sqrt{(\hat V^n_t)^2},
\]
which combined with (\ref{eq:hatMt}), Lemma~\ref{lemma:Dbound} and~(\ref{eq:arrival-A2}) yields
\begin{equation}
\limsup_{n\to\infty} \expect_{\pi_n} (\hat Y^n_t)^- < \infty. \label{eq:PR100e2}
\end{equation}
Now, in view of $\tilde p_1=1$ we have
\begin{align*}
\hat Y^n_t &=\tilde{\vct  p}\cdot \bar {\vct  Y}^n_t-\sum_{2\le k\le K}\tilde p_k \hat Y^n_{t-k} \\
&\le\tilde{\vct  p}\cdot \bar {\vct  Y}^n_t+\sum_{2\le k\le K}\tilde p_k (\hat Y^n_{t-k})^-,
\end{align*}
and it then follows from (\ref{eq:PR100e1}) and (\ref{eq:PR100e2}) that
\begin{align*}
\limsup_{n\to\infty} \expect_{\pi_n} \hat Y^n_t<\infty.
\end{align*}
This bound together with (\ref{eq:PR100e2}) and Markov inequality implies the tightness of the
sequence $\{\hat Y^n_t, n\geq 1\}$ with respect to the sequence of distributions $\{\pi_n, n\geq 1\}$.
\end{proof}

For the purposes of the proof of Theorem~\ref{theorem:MainResultTightness}  it is convenient to
define a sequence of {\em stationary} random processes
\[
\left\{\hat{\vct \Upsilon}^n_t = \left(\hat Q^n_t,\hat{\vct L}^n_t, \hat A^n_t,
\hat {\vct \D}^n_t, \hat \D^n_t, \hat Y^n_t, a^n_t \right), \, t \in\rR_+ \right\}
\]
indexed by $n$. Assume that  $(\hat Q^n_t, \hat {\vct L}^n_t, a^n_t)$ (or equivalently the
extended process $\hat{\vct \Upsilon}^n_t$)
is  distributed according to $\pi_n$ for all $t \in \rR_+$ (see
Section~\ref{section:MainResults}).

\begin{coro}\label{coro:ProcessTight}
For a fixed $t \geq 0$, the sequence $\{\hat{\vct \Upsilon}^n_t, n \geq 1\}$ is tight with respect to the sequence of probability measures~$\{\pi_n, n\geq 1\}$.
\end{coro}

\begin{proof}
The tightness of random variables $\{\hat A^n_t, n\geq 1\}$ follows from (\ref{eq:arrival-A2}) and the tightness of $\{\hat Y^n_t, n\geq 1\}$ is due to
Proposition~\ref{prop:YthetafinitePrelim}. The tightness of $\{\hat
Q^n_t, n\geq 1\}$ then follows from (\ref{eq:Y=Q-beta}), and, thus,
(\ref{eq:scaledDynamics}) implies the tightness of $\{\hat \D^n_t,
n\geq 1\}$. The tightness of $\{\hat \D^n_t, n\geq 1\}$ implies
via~(\ref{eq:hatD}) that $\|{\vct \D}^n_t\|/(\mu n)\rightarrow 1$
with probability~1. Recalling that $\hat {\vct \D}^n_t$ conditional on
$\hat {\D}^n_t$ is independent from all the other r.v.s (Corollary~\ref{coro:MCscaled}), we obtain tightness of the
sequence $\{\hat {\vct \D}^n_t, n\geq 1\}$.  Finally, applying
iteratively~(\ref{eq:Lnhat}) of Corollary~\ref{coro:MCscaled}, we
obtain the tightness of $\hat L^n_{t,K}$, $\hat L^n_{t,K-1}$ and
$\hat L^n_{t,1}$. Tightness of $a^n_t$ follows from the equilibrium assumption of the arrival processes,
which implies that $\E a^n_t = 
(c_{a,n}^2 + 1)/(2\lambda_n) = O(1/n)$, as $n\to\infty$, due to our
assumption $c_{a,n}\rightarrow c_a<\infty$ as $n\rightarrow\infty$.
This completes the proof of the corollary.
\end{proof}

The preceding result implies the weak convergence of $\pi_n$ along some
subsequence $\{n_k, k\ge 1\}$ to some limiting probability measure $\pi_*$~\cite[p.~59]{Bil99}. For now let $\pi_*$ be any such limit measure. Later in this section we establish the uniqueness of $\pi_*$.
Observe that the tightness of $\{{\vct \Upsilon}^n_t, n \geq 1\}$ implies the tightness
of $\{({\vct \Upsilon}^n_t,{\vct \Upsilon}^n_{t+1}), n \geq 1\}$.

\begin{prop}\label{prop:independence} Let $\{{\vct \Upsilon}^n_t, t \in \rZ_+\}$
be in stationarity and suppose
$({\vct \Upsilon}^n_t, {\vct \Upsilon}^n_{t+1}) \Rightarrow (\check{\vct \Upsilon}_t, \check{\vct \Upsilon}_{t+1})$, as $n\to\infty$,
for some $(\check{\vct \Upsilon}_t, \check{\vct \Upsilon}_{t+1})$, where $\check{\vct \Upsilon}_t =
\left(\check Q_t,\check{\vct L}_t, \check A_t, \check {\vct \D}_t, \check \D_t, \check Y_t, \check a_t \right)$. Then the r.v.s $\check A_{t+1}$ and $\check{\vct \Upsilon}_t$ are  independent.
\end{prop}

\begin{proof}
By~(\ref{eq:arrival-Adistr}) and~(\ref{eq:HW}) it follows that $\check A_{t+1}$ is equal in distribution to $\hat A_{t+1}$. One needs to show that for every real $a$ and $\vct b$
\begin{equation}
\pr[\check A_{t+1}\le a, \check{\vct\Upsilon}_t\le {\vct b}] =
\Pr[\check A_{t+1}\le a] \,\pr[\check{\vct\Upsilon}_t\le {\vct b}],\label{eq:independence0}
\end{equation}
where for the vector case ``$\le$" is interpreted coordinate-wise. Given a multidimensional
r.v. $\vct X$, recall that a vector $\vct x$ is defined
to be a continuity point if $\pr[X_i=x_i]=0$ for every coordinate~$i$; it is known that the set of continuity points is a dense
uncountable set (see~\cite[Sect.~2.9]{Dur05}). Since
distribution functions are right-continuous, it suffices to
establish the identity (\ref{eq:independence0}) for the case when $a$ and $\vct b$ are
continuity points of $\check A_{t+1}$ and ${\vct\Upsilon}_t$
respectively, as in this case, by density property, we can find a
sequence of continuity points $(a_n, {\vct b}_n) \downarrow (a,{\vct
b})$ as $n \to \infty$. Thus, one needs to establish~(\ref{eq:independence0}) with $a$ and ${\vct b}$ being continuity points.

The key to the proof is the observation that, conditional on the
backward recurrence time~$a^n_t$, r.v.s $\hat A^n_{t+1}$ and $\hat{\vct \Upsilon}^n_t$ are independent, i.e.,
\begin{equation*}
\pr_{\pi_n}[\hat A^n_{t+1}\le a, \hat{\vct\Upsilon}^n_t \le {\vct b}] =
\int_0^\infty \pr[\hat A^n_{t+1} \le a\,|\, a^n_t=z]\pr_{\pi_n}
[\hat{\vct\Upsilon}^n_t \le {\vct b}\, |\, a^n_t=z] \, d\pr[a^n_t \leq z]. 
\end{equation*}
By assumption (\ref{eq:arrival-Adistr}) we have
\begin{align*}
\sup_{z\ge 0} \Big|\pr[\hat A^n_{t+1} \le a\,|\, a^n_t=z] - \pr[\hat A_{t+1}\le a]\Big| \le \epsilon,
\end{align*}
for all sufficiently large $n$. Therefore, for all such $n$ the following holds
\begin{align*}
\pr_{\pi_n}[\hat A^n_{t+1}\le a, \hat{\vct\Upsilon}^n_t \le {\vct b}]
&\le (\pr[\check A_{t+1}\le a]+\epsilon)\int_0^\infty \pr_{\pi_n}
[\hat{\vct\Upsilon}^n_t \le {\vct b}\, |\, a^n_t=z] \, d\pr[a^n_t \leq z] \\  
&\le \pr[\check A_{t+1}\le a] \, \pr_{\pi_n}[\hat{\vct\Upsilon}^n_t \leq {\vct b}] + \epsilon.
\end{align*}
Recall that $\vct b$ is a continuity point of $\hat{\vct \Upsilon}^n_t$.
Then the weak convergence  $\hat{\vct\Upsilon}^n_t\Rightarrow
\check{\vct\Upsilon}_t$ implies $\pr_{\pi_n}[\hat{\vct\Upsilon}^n_t\le {\vct b}] \to
\pr[\check{\vct\Upsilon}_t\le {\vct b}]$ as $n\to\infty$, resulting in
\begin{align*}
\limsup_{n\rightarrow\infty}\pr_{\pi_n}[\hat A^n_{t+1}\le a, \hat{\vct\Upsilon}^n_t \le {\vct b}]
\le \pr[\check A_{t+1}\le a] \, \pr[\check{\vct\Upsilon}_t \leq {\vct b}] + \epsilon.
\end{align*}
Similarly we establish
\begin{align*}
\liminf_{n\rightarrow\infty}\pr_{\pi_n}[\hat A^n_{t+1}\le a, \hat{\vct\Upsilon}^n_t \le {\vct b}]
\ge \pr[\check A_{t+1}\le a] \, \pr[\check{\vct\Upsilon}_t \leq {\vct b}] - \epsilon.
\end{align*}
On the other hand, by the assumed weak convergence one has $\Pr_{\pi_n}[\hat A^n_{t+1}\le
a,\hat{\vct\Upsilon}^n_t\le {\vct b}] \to \pr[\check A_{t+1}\le a,\check{\vct\Upsilon}_t\le {\vct b}]$ as $n\to\infty$ since $(a,{\vct b})$ is a continuity point of the vector $(\check A_{t+1},\check{\vct\Upsilon}_t)$. Parameter $\epsilon$
is arbitrary and, hence, the assertion of the proposition follows.
\end{proof}

We developed the necessary tools for proving
Theorem~\ref{theorem:MainResultTightness}.  In the proof, we show the existence of
a weak subsequential limit of $\hat{\vct \Upsilon}^n_t$, as
$n\to\infty$, that must correspond to the stationary distribution of
the Markov chain corresponding to~(\ref{eq:L'}), (\ref{eq:Q'})
and~(\ref{eq:D'}) (Section~\ref{section:MainResults}). In the second
part of the proof we argue that a stationary distribution of this
Markov chain is unique.

\begin{proof}[Proof of Theorem~\ref{theorem:MainResultTightness}]
{\em (Part I.)} By Corollary~\ref{coro:ProcessTight} there exists a
subsequence $\{n_k, k\ge 1\}$ along which a weak convergence $(\hat{\vct
\Upsilon}^{n_k}_t, \hat{\vct \Upsilon}^{n_k}_{t+1}) \Rightarrow (\check {\vct
\Upsilon}_t, \check {\vct \Upsilon}_{t+1})$ as $k\to\infty$ takes
place~\cite[Sect.~2.2]{Dur05} for a fixed $t$ and a pair of random
vectors $\check {\vct \Upsilon}_\cdot = \left(\check Q_\cdot,
\check{\vct L}_\cdot, \check A_\cdot, \check {\vct \D}_\cdot, \check
\D_\cdot, \check Y_\cdot, \check a_\cdot \right)$. The Continuous Mapping
Theorem~\cite[Sect.~2]{Bil99} yields the following weak limits
along $\{n_k, k\ge 1\}$:
\begin{align*}
\left(\hat Q^{n_k}_t + \hat A^{n_k}_{t+1} + \sum_{j=2}^K \hat L^{n_k}_{t,j} - \beta_{n_k} \right)^+
&\Rightarrow \left(\check Q_t + \check A_{t+1} + \sum_{j=2}^K \check L_{t,j} - \beta \right)^+, \\
\left(\hat Q^{n_k}_t + \hat A^{n_k}_{t+1}\right) \wedge \left(\beta_{n_k} - \sum_{j=2}^K \hat L^{n_k}_{t,j} \right)
&\Rightarrow \left(\check Q_t + \check A_{t+1} \right) \wedge \left(\beta - \sum_{j=2}^K \check L_{t,j}\right).
\end{align*}
Then from the preceding and Corollary~\ref{coro:MCscaled} the
following  relations follow for the elements of $\check{\vct
\Upsilon}_t$ and~$\check{\vct \Upsilon}_{t+1}$:
\begin{align*}
\check {\vct  L}_{t+1} &= {\cal T}\{\check {\vct  L}_t\} + \check {\vct \D}_{t+1} + \check \D_{t+1} \vct {p}, \\ 
\check Q_{t+1} &= \left(\check Q_t + \check A_{t+1} + \sum_{j=2}^K \check L_{t,j} - \beta \right)^+,\\ 
\check \D_{t+1} &= \left(\check Q_t + \check A_{t+1} \right) \wedge \left(\beta - \sum_{j=2}^K \check L_{t,j}\right). 
\end{align*}
Now, note that $\hat A^n_t \Rightarrow \hat A_t$ and  $\hat {\vct
\D}^n_t \Rightarrow \hat {\vct \D}_t$ as $n \to\infty$ for every
$t\in\rZ_+$. These weak limits are due to central limit theorems for
renewal processes~\cite[p.~154]{Bil99} and vectors in $\rR^K$~\cite[p.~385]{Bil95}, respectively. Moreover, $\hat A_t$ and $\hat{\vct
\D}_t$ are independent in addition to the independence of $\check
A_{t+1}$ and~$\check {\vct \Upsilon}_t$ (see
Proposition~\ref{prop:independence}). Since $\pi_n$ is the
stationary distribution of $\hat{\vct\Upsilon}^n_t$, we obtain that
the distribution of $\check{\vct\Upsilon}_t$ coincides with a stationary
distribution of the Markov chain specified by~(\ref{eq:L'}), (\ref{eq:Q'}) and~(\ref{eq:D'}).

\vspace{.025in}
{\em (Part II.)} We established in the previous part
that every weak subsequential  limit $\check{\vct\Upsilon}_t$ of $\hat{\vct\Upsilon}^n_t$
is a stationary distribution of the Markov chain $\hat{\vct\Upsilon}_t$ defined by (\ref{eq:L'}), (\ref{eq:Q'}) and (\ref{eq:D'}).
It remains to establish the uniqueness of the
stationary measure $\pi_*$ of $\{\hat{\vct\Upsilon}_t, t\in\rZ_+\}$. The uniqueness
of the limit measure implies also the convergence $\pi_n\Rightarrow \pi_*$, using standard results of weak convergence theory~\cite[p.~59]{Bil99}.

The proof of uniqueness uses the framework of  Harris
chains and Harris recurrence. All of the definitions and results are
adopted from \cite[Ch.~5]{Dur05}. Recall that the Markov chain
$\{(\hat Q_t,\hat{\vct L}_t), t\in\rZ_+\}$ is a Harris chain if one
can identify two (measurable) sets ${\cal A}$, ${\cal B}\subset \rR^{K+1}$ and a
probability measure $\nu$ concentrated on ${\cal B}$ such that for
every ${\vct x} \in \rR^{K+1}$
\[
\sum_{t\ge 0} \Pr[(\hat Q_t,\hat{\vct L}_t) \in {\cal A} \,|\, (\hat Q_0,\hat{\vct L}_0) = {\vct x}]>0,
\]
and there exists $\epsilon>0$ such that for every ${\cal C} \subset {\cal B}$
\begin{equation}
\inf_{x\in {\cal A}} \Pr[(\hat Q_{t+1},\hat{\vct L}_{t+1}) \in {\cal C}\,|\,(\hat Q_t,\hat{\vct L}_t)=
{\vct x}] \geq \epsilon \nu({\cal C}). \label{eq:Harris}
\end{equation}
Moreover, if these conditions hold for some ${\cal B}={\cal A}$, and the Markov chain possesses a stationary distribution, then the Markov chain is also mixing, and as a result
the stationary distribution is unique (see \cite[Theorem~6.8]{Dur05}
and the comment on aperiodicity just preceding it). Note that if
$\pi$ is a stationary distribution of $\{(\hat Q_t,\hat{\vct L}_t),
t\in\rZ_+\}$ then $\pi$ is also a stationary distribution of
$\{(\hat Q_{2Kt},\hat{\vct L}_{2Kt}), t\in\rZ_+\}$. In view of this,
(\ref{eq:Harris}) can be replaced by
\begin{equation}
\inf_{x\in {\cal A}} \Pr[(\hat Q_{t+2K},\hat{\vct L}_{t+2K})
\in {\cal C}\,|\,(\hat Q_t,\hat{\vct L}_t)= {\vct x}] \geq \epsilon \nu({\cal C}). \label{eq:Harris2}
\end{equation}

Thus, our task of proving the uniqueness of the stationary distribution $\pi_*$
is reduced to constructing the set ${\cal A}={\cal B}$ satisfying the assumptions above. For this purpose we set
\begin{align*}
{\cal A} = \left\{{\vct x}\in \rR^{K+1}: x_1=0, \bigvee_{j=2}^{K+1} |x_j| < \beta/K^2 \right\}.
\end{align*}
Namely, $(\hat Q_t, \hat {\vct L}_t) \in {\cal A}$ implies that the
queue length $\hat Q_t$ is equal to 0 and each $\hat L_{t,j}$,
$1\leq j \leq K$, is upper bounded bounded by $\beta/K^2$ in
absolute value. We set ${\cal B}={\cal A}$ and claim that ${\cal A}$
satisfies the requirements when $\nu$ is the uniform distribution on
${\cal A}$. For a pair of positive constants $c,C$ define an event
${\cal U}$ by
\[
{\cal U} = \left\{\bigvee_{i=1}^K \hat A_{t+i} < -C \right\} \cap \left\{\bigvee_{i=K+1}^{2K} |\hat A_{t+i}| < c \right\}
\]
and note that $\Pr[{\cal U}]>0$ due to the Gaussian nature of $\hat A_t$'s.

First, we show that $\Pr[(\hat Q_{t+2K},\hat{\vct L}_{t+2K}) \in
{\cal A} \,|\, (\hat Q_t,\hat{\vct L}_t)={\vct x}]>0$ for every
$\vct x$.  To this end, given~(\ref{eq:L'}), (\ref{eq:Q'}),
(\ref{eq:D'}) and $(\hat Q_t,\hat{\vct L}_t)={\vct x}$, there exists
$C$ large enough so that
\begin{equation}
\Pr\left[\hat Q_{t+K}=0,  \bigvee_{i=1}^K \hat L_{t+K,i} < -\beta \, \bigg| \,{\cal U}\right] > 0. \label{eq:25jul002}
\end{equation}
To verify this claim note that~(\ref{eq:L'}) implies
\begin{equation}
\hat L_{t+K,i} = \sum_{j=i}^K \left(\hat \D_{t+K+i-j,j} + p_j \hat \D_{t+K+i-j} \right), \label{eq:25jul001}
\end{equation}
and that $\Pr[\vee_{i,j=1}^K | \hat \D_{t+i,j}| \leq \epsilon ] >0$ for any $\epsilon>0$ due to the normal distribution. Then, by selecting $C>p_K^{-1} (\hat Q_t + \sum_{i=1}^K |\hat L_{t,i}| + \beta +\epsilon K)$ and $\epsilon$ small enough, recursions~(\ref{eq:Q'}) and~(\ref{eq:D'}) render, on the event ${\cal U} \cap \{\vee_{i,j=1}^K | \hat \D_{t+i,j}| \leq \epsilon\}$, to
\begin{align*}
\hat Q_{t+1} &= 0,\\
\hat \D_{t+1} &= \hat Q_t + \hat A_{t+1} \leq -p_K^{-1} (\beta+\epsilon K),
\end{align*}
leading to $\hat L_{t+1,K} \leq \epsilon -\beta-\epsilon K$ by~(\ref{eq:25jul001}).
Next, on the same event $\hat Q_{t+2}=0$,
$\hat \D_{t+2} = \hat A_{t+2} \leq -p_K^{-1}(\beta+\epsilon K)$,
$\hat L_{t+2,K} \leq - \beta -\epsilon (K-1)$ and $\hat L_{t+2,K-1} \leq -\beta -\epsilon (K-2)$.
Further iteration over the time index and~(\ref{eq:25jul001}) yield~(\ref{eq:25jul002}).

In addition, for $c$ small enough in the definition of $\cal U$, on
event $\{\hat Q_{t+K}=0, \vee_i \hat L_{t+K,i} <-\beta\}$, we have $Q_{t+K+i}=0$ and $\hat J_{t+K+i}=\hat A_{t+K+i}$ for $i=1,\ldots,K$ by a similar argument as above. Then the components $2,\ldots,K+1$ of $(\hat Q_{t+2K},\hat{\vct L}_{t+2K})$ are bounded in absolute value by $\beta/K^2$ provided that
\begin{align}
\hat\D_i' = \sum_{j=i}^{K} \hat\D_{t+2K+i-j,j} \in \left\{ [-\beta/K^2,\beta/K^2] - \sum_{j=i}^{K} p_j \hat A_{t+2K+i-j} \right\}, \label{eq:eventD}
\end{align}
for $i=1,\ldots,K$. We denote by $\cal E$ the conjunction of $\cal
U$ and  the event described by~(\ref{eq:eventD}). Recall that
$\{\hat{\vct \D}_t, t \in \rZ_+\}$ is an i.i.d. sequence of
multivariate Gaussian random vectors, independent from all other
r.v.s, with the covariance matrix
$\Sigma$~(\ref{eq:Sigma}). Thus, $(\hat J_1',\ldots,\hat J_K')$ is a
zero-mean multivariate Gaussian vector with $\expect \hat J'^2_i =
\sum_{j=i}^K (1-p_j)p_j$ and $\expect \hat J'_i \hat J'_j = -
\sum_{k=j}^K p_{k+i-j} p_k$, $i<j$.
In particular it has a continuous positive density everywhere on $\rR^K$.
Namely, assume that $(\hat J'_i, \hat J'_{i+1},\ldots,\hat J'_{K})$ has
a continuous positive density everywhere on $\rR^{K+1-i}$; this assumption
holds for $i=K$ because $p_K>0$. Then, $(\hat J'_{i-1}, \hat J'_i,\ldots,\hat J'_{K})$
has a continuous density everywhere on $\rR^{K+2-i}$ since
\[
\hat\D_{i-1}' = \hat\D_{t+K+i-1,K} + \sum_{j=i-1}^{K-1} \hat\D_{t+2K+i-1-j,j},
\]
$\hat\D_{t+K+i-1,K}$ is independent of $\{\hat {\vct \D}_{t+K+j}, j=i,\ldots,K \}$
and $(\hat J'_i, \hat J'_{i+1},\ldots,\hat J'_{K})$ is a deterministic function of  $\{\hat {\vct \D}_{t+K+j}, j=i,\ldots,K \}$.
Clearly then $\Pr[{\cal E} \,|\, (\hat Q_t,\hat{\vct L}_t) = {\vct x}]>0$ for every $\vct x$.

Second, as in the preceding, by continuity and strict positivity of
the density  of $\hat A_t$ and $\hat{ \D}'_i$, there exists
$\alpha>0$ such that for every set ${\cal C} \subset {\cal A}$
\[
\inf_{{\vct x} \in {\cal A}} \Pr[(\hat Q_{t+2K},\hat{\vct L}_{t+2K}) \in {\cal C}\,|\,(\hat Q_t,\hat{\vct L}_t)={\vct x}] \geq \alpha \nu({\cal C}),
\]
and the requirement (\ref{eq:Harris2}) holds. Thus $\{(\hat
Q_t,\hat{\vct L}_t), t\in \rZ_+\}$  is indeed a Harris chain which
admits a unique stationary distribution. This completes the proof.
\end{proof}

\section{Proof of Theorem~\ref{theorem:MainResultDecayRate}}
\label{sec:Thm2}

This section is devoted to proving our second main result,
Theorem~\ref{theorem:MainResultDecayRate}.  The approach is based on
the results of Section~\ref{section:LyapunovFunction} for the
limiting Markov chain $\{\hat{\vct\Upsilon}_t,\, t \in\rZ_+\}$ in
steady state. The proof utilizes the following preparatory lemma. The operators ``$\leq$" and ``$\geq$" are interpreted element-wise.

\begin{lemma} \label{lemma:Gammabound} Let
\[
\Gamma^T = \begin{bmatrix}
\vct p \\
I \,\,\,\, \vct 0^T
\end{bmatrix},
\]
where $I$ is the $(K-1)\times (K-1)$ identity matrix and $\vct 0$ is a $(K-1)$-dimensional vector of zeros. Then for $t\geq K-1$ and $k\geq 0$
\[
- \bar {\vct \fM}_{t+k} - \beta {\vct B}_{k} \leq \bar {\vct Y}_{t+k} - (\bar{\vct Y}_t)^+ \,\Gamma^k \leq \bar {\vct \fM}_{t+k},
\]
where ${\vct B}_k = (k,(k-1)^+,\ldots,(k-K+1)^+)$ and
\[
\bar {\vct \fM}_t = \left(\sum_{i=t-K+1}^t |\hat V_i|, \sum_{i=t-K+1}^{t-1} |\hat V_i|, \ldots, \sum_{i=t-K+1}^{t-K+1} |\hat V_i| \right).
\]
\end{lemma}

\begin{remark} \label{remark:Gamma}
Note that $\Gamma^T$ is an irreducible, aperiodic stochastic matrix since $\|{\vct p}\| =1$, $p_K>0$ and there exist relatively prime $i$ and $j$ such that $p_i p_j >0$ (see Section~\ref{sec:GGN}). Therefore, $\Gamma^k \to (\psi^T,\ldots,\psi^T)$ as $k \to \infty$ for some unique probability vector $\psi$.
\end{remark}

\begin{proof}
The proof is by induction over $k$. First, we claim that the statement holds for $k=0$:
\[
- \bar {\vct \fM}_t - \beta {\vct B}_0 \leq \bar {\vct Y}_t - (\bar{\vct Y}_t)^+ \leq \bar {\vct \fM}_t,
\]
or in the scalar form
\[
-\sum_{i=t-K+1}^{t-j} |\hat \fM_i| \leq \hat Y_{t-j} - (\hat Y_{t-j})^+ \leq \sum_{i=t-K+1}^{t-j} |\hat V_i|,
\]
where $j=0,1,\ldots,K-1$. The upper bound is trivial due to the nonnegativity of $|\hat \fM_i|$ for all $i$; the same holds for the lower bound when $\hat Y_{t-j} \geq 0$. The case $\hat Y_{t-j} < 0$ is covered by Lemma~\ref{lemma:Ytsummary}(i) since it implies $\hat Y_{t-j} \geq \hat \fM_{t-j}$. Now, assume that the statement holds for some~$k$ and note that
\begin{align*}
(\bar {\vct Y}_t)^+ \,\Gamma &= \left(\sum_{i=1}^{K} p_i \hat Y_{t+1-i}^+, \hat Y_t^+, \hat Y_{t-1}^+,\ldots,\hat Y_{t-K+2}^+ \right), \\
\bar {\vct \fM}_t \,\Gamma &\leq \left(\sum_{i=t-K+1}^t |\hat V_i|, \sum_{i=t-K+1}^t |\hat V_i|, \sum_{i=t-K+1}^{t-1} |\hat V_i|, \ldots, \sum_{i=t-K+1}^{t-K+2} |\hat V_i| \right).
\end{align*}
Consider the upper bound first. The preceding two relationships, Lemma~\ref{lemma:Ytsummary}(i) and the inductive assumption yield
\begin{align*}
\bar {\vct Y}_{t+k+1} &\leq (\bar {\vct Y}_{t+k})^+ \,\Gamma + (|\hat V_{t+k+1}|,0,\ldots,0) \\
&\leq (\bar {\vct Y}_t)^+ \,\Gamma^{k+1} + \bar {\vct \fM}_{t+k} \Gamma + (|\hat V_{t+k+1}|,0,\ldots,0) \\
&\leq (\bar {\vct Y}_t)^+ \,\Gamma^{k+1} + \bar {\vct \fM}_{t+k+1},
\end{align*}
where $(x-\beta)^+ \leq x^+$ is also used. As far as the lower bound is concerned, the same arguments and $(x-\beta)^+ \geq x^+ -\beta$ result in
\begin{align*}
\bar {\vct Y}_{t+k+1} &\geq (\bar {\vct Y}_{t+k})^+ \, \Gamma - (|\hat V_{t+k+1}| + \beta,0,\ldots,0) \\
&\geq (\bar {\vct Y}_t)^+ \, \Gamma^{k+1} - \bar {\vct \fM}_{t+k} \Gamma - \beta {\vct B}_{k} \Gamma - (|\hat V_{t+k+1}| + \beta,0,\ldots,0) \\
&\geq (\bar {\vct Y}_t)^+ \, \Gamma^{k+1} - \bar {\vct \fM}_{t+k+1} - \beta {\vct B}_{k+1}. \qedhere
\end{align*}
\end{proof}

We conclude this section with the proof of Theorem~\ref{theorem:MainResultDecayRate}.

\begin{proof}[Proof of Theorem~\ref{theorem:MainResultDecayRate}]
Proposition~\ref{prop:LyapunovBoundLimit}, Lemma~\ref{lemma:inRboundLimit} and
Theorem~\ref{theorem:BoundsStationaryGeometric}, where in the theorem each $\Xi^n$ is identified with $\{\hat{\vct\Upsilon}_t, t \in\rZ\}$, $\pi_n=\pi_*$ and $\mathcal{R}_{\beta_n}=\mathcal{R}_{\beta}$, yield
\begin{equation}
\expect_{\pi_*} \Phi_\theta (\bar {\vct  Y}_t, \bar {\vct  Z}_t) < \infty \label{eq:T2Pe1000}
\end{equation}
for every $\theta<\theta^*/\mu$. On the other hand, taking expectation (with respect to $\pi_*$) of both sides of~(\ref{eq:PhiincreaseLimit}) implies
\begin{equation}
\expect_{\pi_*} \Phi_\theta (\bar {\vct  Y}_t, \bar {\vct  Z}_t) = \infty \label{eq:T2Pe1010}
\end{equation}
for every $\theta>\theta^*/\mu$.

Next, the definition~(\ref{eq:PhiDef}) of $\Phi_\theta$ renders $\tilde{\vct  p} \cdot \bar {\vct  Y}_t = \theta^{-1}\log\Phi_\theta (\bar {\vct  Y}_t,\bar {\vct  Z}_t) - {\vct\alpha} \cdot \bar {\vct  Z}_t$. This equality, (\ref{eq:T2Pe1000}), (\ref{eq:T2Pe1010}), the normal distribution of $\bar {\vct Z}_t$ and Proposition~\ref{prop:productDthetaDinfty} of the appendix result in
\begin{equation}
\E_{\pi_*} \exp\{\theta\tilde{\vct  p} \cdot \bar {\vct  Y}_t\} < \infty \label{eq:T2Pe1020}
\end{equation}
for every $\theta<\theta^*/\mu$ while
\begin{equation}
\E_{\pi_*}\exp\{\theta\tilde{\vct  p} \cdot \bar {\vct  Y}_t\} = \infty \label{eq:T2Pe1030}
\end{equation}
for every $\theta>\theta^*/\mu$. Given~(\ref{eq:T2Pe1020}) and~(\ref{eq:T2Pe1030}), in order to complete the proof of the theorem it is sufficient to prove for every $\theta>0$
\begin{equation}
\expect_{\pi_*}  \exp\{\theta |\mu^{-1}\hat Y_t-\tilde{\vct  p} \cdot \bar {\vct  Y}_t| \} < \infty, \label{eq:T2Pe1040}
\end{equation}
or equivalently $|\mu^{-1}\hat Y_t-\tilde{\vct  p} \cdot \bar {\vct  Y}_t| \in \mathcal{M}_\infty$ assuming the stationarity of $\{\hat Y_t,\, t\in \rZ\}$. Informally,~(\ref{eq:T2Pe1040}) implies that the stationary r.v.s $\mu^{-1}\hat Y_t$ and $\tilde{\vct  p} \cdot \bar {\vct  Y}_t$ have the same exponential decay rate.

The rest of the proof is devoted to establishing~(\ref{eq:T2Pe1040}). Given that $\mu^{-1}\hat Y_t=\sum_k\tilde p_k \hat Y_{t}$, by Proposition~\ref{prop:DinftyLinComb} of the appendix
it suffices to show that $|\hat Y_{t}-\hat Y_{t-k}| \in \mathcal{M}_\infty$ for every $k=1,\ldots,K~-~1$ and stationary $\{\hat Y_t, t\in\rZ\}$. Consider an arbitrary such $k$ and note that Lemma~\ref{lemma:Gammabound} renders, for $j \geq 1$ and $t \geq K-1$,
\[
-\bar {\vct V}_{t+j} - \beta {\vct B}_j \leq \bar {\vct Y}_{t+j} - (\bar {\vct Y}_t)^+ \Gamma^j \leq \bar {\vct V}_{t+j}.
\]
Rewriting the preceding relationship in a scalar form renders
\begin{align*}
- \sum_{i=t-K+1}^{t+j} |\hat \fM_i| - j\beta &\leq \hat Y_{t+j} - \sum_{i=0}^{K-1} (\Gamma^j)_{i+1,1} \hat Y_{t-i}^+ \leq \sum_{i=t-K+1}^{t+j} |\hat \fM_i|, \\
- \sum_{i=t-K+1}^{t+j-k} |\hat \fM_i| - (j-k)^+\beta &\leq \hat Y_{t+j-k} -  \sum_{i=0}^{K-1} (\Gamma^{j})_{i+1,k+1} \hat Y_{t-i}^+ \leq \sum_{i=t-K+1}^{t+j-k} |\hat \fM_i|,
\end{align*}
and, hence,
\begin{equation}
|\hat Y_{t+j}-\hat Y_{t+j-k}| \leq \sum_{i=0}^{K-1} |(\Gamma^j)_{i+1,1} - (\Gamma^{j})_{i+1,k+1}| \hat Y_{t-i}^+ + 2\sum_{i=t-K+1}^{t+j} |\hat \fM_i| + 2(j+K+1)\beta. \label{eq:T2Pe1080}
\end{equation}
In view of Remark~\ref{remark:Gamma}, the speed of convergence of $\Gamma^k$ in $k$ is exponential~\cite[p.~211]{Bre99}, i.e., there exist constants $C$ and $\gamma < 1$ such that
\[
\sup_{0 \leq i \leq K-1} |(\Gamma^j)_{i+1,1} - (\Gamma^{j})_{i+1,k+1}| \leq C \gamma^j.
\]
Then, (\ref{eq:T2Pe1080}) and the preceding inequality yield
\begin{equation}
|\hat Y_{t+j}-\hat Y_{t+j-k}| \leq C \gamma^j \sum_{i=0}^{K-1} \hat Y_{t-i}^+ + 2\sum_{i=t-K+1}^{t+j} |\hat \fM_i| + 2(j+K+1)\beta.  \label{eq:T2Pe1090}
\end{equation}
Now, observe that the last two terms on the right-hand side of the
preceding inequality are elements of ${\cal M}_\infty$ due
to~(\ref{eq:MZLimit}), (\ref{eq:ZDinftyLimit}), and
Proposition~\ref{prop:DinftyLinComb}. In addition, from $\hat Y_t =
\tilde{\vct p} \cdot \bar {\vct Y}_t - \sum_{i=2}^K \tilde p_i \hat
Y_{t-i+1}$ (Lemma~\ref{lemma:Ytsummary}(i)), (\ref{eq:T2Pe1020}),
Lemma~\ref{lemma:Ytsummary}(ii) and
Proposition~\ref{prop:DinftyLinComb}, it follows that $\hat
Y_{t}\in\mathcal{M}_{\theta'}$ for some sufficiently small $\theta'
>0$. By stationarity of $\{\hat Y_t, t\in \rZ_+\}$ this applies to
every term in the first sum on the right-hand side
of~(\ref{eq:T2Pe1090}). It then follows that $|\hat Y_{t}-\hat
Y_{t-k}|\in\mathcal{M}_{\theta''}$  with
$\theta''=\gamma^{-j} \theta'/(CK)$. Since $j$ is arbitrary, by taking it sufficiently
large we establish $|\hat Y_{t}-\hat Y_{t-k}|
\in\mathcal{M}_\infty$. This concludes the proof
of~(\ref{eq:T2Pe1040}) and the proof of the theorem.
\end{proof}

\section{Proof of Theorem~\ref{theorem:MainResultExpErg}}
\label{sec:thm3}
Proposition~\ref{prop:LyapunovBoundn}, the second part of Lemma~\ref{lemma:inRboundLimit} and Theorem~\ref{theorem:BoundsStationaryGeometric} from the Appendix imply the statement of the theorem.

\section{Proof of Corollary~\ref{coro:Waiting}}
\label{sec:CorollaryWaiting}
First, we note that for $x \geq 0$, as $n\to\infty$,
\begin{equation}
A^n_{0,x/\sqrt{n}}/\sqrt{n} \to \mu x \label{eq:arrival-AdistrScaleRootn}
\end{equation}
in probability. Let $\{\tau_{n,i}, i\geq 1\}$ be interarrival times in the $n$th system, with $\tau_{n,1}$ being the time of the first arrival after time $t=0$. The limit is based on the following: (i) $\{A_{0,t} \geq k\} = \{\sum_{i=1}^k \tau_{n,i} \leq t\}$ for $t\geq 0$, $k\geq 1$; (ii) for $n$ large enough Markov's inequality yields for $\epsilon>0$
\begin{align*}
\Pr\left[\sum_{i=2}^{\lfloor (\mu x+\epsilon)\sqrt{n} \rfloor} \tau_{n,i} \leq  x/\sqrt{n}\right] &\leq \Pr\left[\sum_{i=2}^{\lfloor (\mu x+\epsilon)\sqrt{n} \rfloor} (\tau_{n,i} - 1/\lambda_n) \leq  -2\epsilon/(\mu \sqrt{n}) \right] \\
& \leq \mu^2 (\mu x+\epsilon) \epsilon^{-2} n^{3/2} \text{Var}(\tau_{n,2}) \to 0
\end{align*}
and, similarly,
\[
\Pr\left[\sum_{i=2}^{\lceil (\mu x-\epsilon) \sqrt{n}\rceil} \tau_{n,i} >  x/\sqrt{n}\right] \to 0
\]
as $n\to\infty$; and (iii) the arrival processes are in stationarity and, thus, $\tau_{n,1}$ has the equilibrium distribution and does not impact~(\ref{eq:arrival-AdistrScaleRootn}).

Second, from the Distributional Little's Law~\cite{HNe71} it follows that $Q^n$ equals in distribution
to the number of arrivals in a renewal process $A^n_t$ during the time interval of length $W^n$ (recall that $\{A^n_t, t\in\rR\}$ is in stationarity), i.e., $Q^n=A^n_{0,W^n}$ in distribution. Then for every $x>0$, the event $\{W^n\le x\}$ implies $\{Q^n\le A^n_{0,x}\}$ and, therefore,
\begin{align*}
\pr_{\pi_n}[\sqrt{n} W^n \le x] &\le \pr_{\pi_n}\left[Q^n \le A^n_{0,x/\sqrt{n}} \right] \\
&=\pr_{\pi_n}\left[Q^n/\sqrt{n} \le A^n_{0,x/\sqrt{n}}/\sqrt{n} \right].
\end{align*}
The distribution of $\hat Q$ is continuous everywhere on $(0,\infty)$ as seen from the presence of $\hat A_{t+1}$ in the expression for $\hat Q_{t+1}$ in (\ref{eq:Q'}). Letting $n\to\infty$ in the preceding inequality and applying~(\ref{eq:arrival-AdistrScaleRootn}) yields
\begin{align*}
\limsup_{n\to\infty} \Pr_{\pi_n} [\sqrt{n} W^n \leq x]
\leq \pr_{\pi_*}[\hat Q\le\mu x].
\end{align*}
Similarly, for every $x>0$, the event $\{W^n>x\}$ implies $\{Q^n\ge A^n_{0,x}\}$, leading to
\begin{align*}
\pr_{\pi_n}[\sqrt{n} W^n>x] &\le \pr_{\pi_n}\left[Q^n\ge A^n_{0,x/\sqrt{n}}\right] \\
&=\pr_{\pi_n}\left[Q^n/\sqrt{n} \ge A^n_{0,x/\sqrt{n}}/\sqrt{n}\right]
\end{align*}
and
\[
\liminf_{n\rightarrow\infty}\pr_{\pi_n}[\sqrt{n}W^n\le x] \ge \pr_{\pi_*}[\hat Q\le \mu x].
\]
The preceding establishes $\pr_{\pi_n}[\sqrt{n}W^n\le x] \to \pr_{\pi_*}[\hat Q\le\mu x]$ as $n\to\infty$ for every $x>0$. The assertion then follows.

\section{Conclusions}
We analyzed a stationary multi-server queue in the Halfin-Whitt
(QED) regime when the service times have a
lattice-valued  distribution with a finite support. Prior analyses
of such systems in steady-state assumed either exponential or
deterministic service times. We described the steady-state distribution of the  appropriately scaled queue length in terms of the steady-state distribution of a
continuous-state Markov chain. One can estimate the steady-state
distribution of this chain either numerically or by simulations. 
Finally, we have established that the large deviations rate of the queue length in steady state is given by $\theta^*=2\beta/(c_a^2+c_s^2)$, where $\beta$ is the extra capacity
parameter of the model and $c_a, c_s$ are the coefficients of
variation of interarrival and service times, respectively.
We conjecture that the expression for $\theta^*$ remains valid for a broad class of service time distributions.

\section*{Acknowledgment}

PM thanks Itay Gurvich and Avishai Mandelbaum for discussions on the QED regime.

\small


\bibliographystyle{plain}

\normalsize

\appendix

\section{Appendix A: Moment generating functions}

Here we list some basic properties of moment generating functions. While these properties are well known, we include the proofs for completeness.

\begin{prop}\label{prop:DinftyLinComb}
Any affine combination of (not necessarily independent) nonnegative
elements of $\mathcal{M}_\infty$ is an element of $\mathcal{M}_\infty$.
\end{prop}

\begin{proof}
Given two sequences $\{X^n_1,n\geq 1\}$, $\{X^n_2, n \geq 1\} \in \mathcal{M}_\infty$ of nonnegative r.v.s,
and reals $a$, $a_1$, $a_2$ and $\theta>0$, the Cauchy-Schwarz inequality implies
\begin{equation}
M^2_{a+a_1X^n_1+a_2X^n_2}(\theta)\le e^{\theta a} M_{X^n_1}(2\theta) \, M_{X^n_2}(2\theta). \label{eq:Mtheta1}
\end{equation}
The definition of $\mathcal{M}_\infty$ and nonnegativity of $X^n_j$, $j=1,2$, render $\limsup_{n\rightarrow \infty} M_{X^n_j}(x)<\infty$, for all $x$, positive or negative. It then follows that the $\limsup$ of the product on the right-hand side of~(\ref{eq:Mtheta1}) is finite.
The proof for the general case is obtained by induction.
\end{proof}

\begin{prop}\label{prop:productDthetaDinfty}
Suppose $\{X^n, n\geq 1\}\in\mathcal{M}_\theta$ for some $\theta>0$ and
$\{Y^n, n\geq 1\}\in\mathcal{M}_\infty$. Then $\{X^n+Y^n, n\geq 1\}\in\mathcal{M}_{\theta'}$ for every $\theta'<\theta$.
\end{prop}

\begin{proof}
Applying H\"{o}lder's inequality $\E[XY]\le (\E[X^p])^{1/p}(\E[Y^q])^{1/q}$ with parameters $p=\theta/\theta'$ and $q=\theta/(\theta-\theta')$, defined from $1/q=1-1/p$, we obtain
\begin{align*}
\E e^{\theta' (X^n+Y^n)} \leq \left(\E e^{\theta X^n} \right)^{\theta'/\theta}
\left(\E e^{\theta Y^n/(\theta-\theta')} \right)^{1-\theta'/\theta}.
\end{align*}
The statement follows.
\end{proof}

\section{Appendix B: Lyapunov functions}

The following definition plays a key role in the proofs of our main results.

\begin{Defi} \label{def:GLF} {\em (Geometric Lyapunov function)}
Let $\Xi = \{\Xi_t,  t\in\rZ_+\}$ be a discrete-time Markov chain
defined on a state space ${\cal X}$, equipped with a $\sigma$-algebra ${\cal F}$.
A function $\Phi:{\cal X}\rightarrow \rR_+$ is defined to be a geometric Lyapunov function for $\Xi$ with a geometric drift size
$0<\delta<1$ and exception set $\mathcal{R}\subset {\cal X}$ if for every $x \in {\cal X}\setminus \mathcal{R}$
\begin{equation*}
\E[\Phi(\Xi_{1})|\Xi_0=x] \leq (1-\delta){\Phi(x)}. 
\end{equation*}

{\em (Quadratic Lyapunov function)}
Under the same setting as above,
a function $\Psi:{\cal X}\rightarrow \rR$ is defined to be a quadratic Lyapunov function for $\Xi$ with exception set $\mathcal{R}\subset {\cal X}$ and parameters $\delta>0$, $0\le \psi<\infty$ if for every $x \in {\cal X}\setminus \mathcal{R}$
\begin{equation*}
\E[\Psi^2(\Xi_{1})|\Xi_0=x]-{\Psi^2(x)} \leq -\delta\Psi(x)+\psi. 
\end{equation*}
\end{Defi}

Informally, the following result shows that if a sequence of Markov chains
admits the same geometric Lyapunov function that is uniformly bounded in expectation in the
exception region, then this function is uniformly bounded in expectation in general.
Our definition of a geometric Lyapunov function as well as the following result is fairly standard~\cite{MTw93,GaZ06}.

\begin{theorem}\label{theorem:BoundsStationaryGeometric}
Let  $\{\Xi^n, n\ge 1\}$ be a sequence of discrete-time Markov chains with ${\cal X}_n$ and $\pi_n$ being the state space and a stationary distribution of $\Xi^n$, respectively. Suppose for every $n \geq 1$ function $\Phi: \cup {\cal X}_n \to \rR_+$ is a geometric Lyapunov function for $\Xi^n$ with drift $\delta$ and exception set $\mathcal{R}_n \subset {\cal X}_n$. If
\begin{equation}
C_{\mathcal{R}}\triangleq\limsup_{n\rightarrow \infty}\E_{\pi_n}\left[\Phi(\Xi^n_1) \ind{\Xi^n_{0}\in\mathcal{R}_n}\right]<\infty \label{eq:CRdef}
\end{equation}
then
\begin{align*}
\limsup_{n\rightarrow \infty} \E_{\pi_n}[\Phi(\Xi^n_1)] \leq C_{\mathcal{R}}/\delta.
\end{align*}
\end{theorem}

\begin{remark} Note that the uniqueness of a stationary distribution $\pi_n$ is not assumed. The theorem holds for  {\em every} sequence of stationary distributions.
\end{remark}

\begin{remark} Our treatment of the geometric Lyapunov function is unconventional. Typically it is assumed that in the exception region the jumps $\Phi(\Xi^n_1)-\Phi(\Xi^n_0)$ are deterministically bounded, e.g., see~\cite{MTw93}. The intuition behind our result is as follows. The expected value of the Lyapunov function is uniformly bounded (in~$n$) since (i) when the chain is in the exception region $\Phi$ is bounded by assumption (in the next time step), and (ii) when the chain is outside of the exception region there is a downward uniform drift decreasing the expected value of $\Phi$.
\end{remark}

\begin{proof}
The proof is similar to the approach taken in~\cite{GaZ06}, and it is based on the Monotone Convergence Theorem.
Assumption~(\ref{eq:CRdef}) implies the existence of $n_0$ such that $\E_{\pi_n}[\Phi(\Xi^n_1) \ind{\Xi^n_0 \in \mathcal{R}_n}]<\infty$ for all $n>n_0$. Fix an arbitrary such  $n$, introduce the following two conditional expectations
\begin{align*}
G^{b} (x) &\triangleq \expect  \left[\Phi(\Xi^n_1)\wedge b \, \big | \,   \Xi^n_0 = x\right], \\
H  (x) &\triangleq \expect \left[\Phi (\Xi^n_1) \ind{\Xi^n_0 \in {\cal R}_{n}} \, \big | \,   \Xi^n_0 = x \right]
\end{align*}
and let $G(x) = G^{\infty}(x)$ for notational simplicity. Then, by the Lyapunov nature of $\Phi$,
the difference of $G (x)$ and $\Phi (x)$ for $x \in {\cal X}_n$ can be bounded as
\[
G (x) - \Phi (x) \leq
\begin{cases}
-\delta  \Phi (x), & x \in {\cal X}_n \setminus{\cal R}_{n}, \\
H (x) - \Phi (x), & x \in {\cal R}_{n},
\end{cases}
\]
the second case being in fact the equality. Due to the nonnegativity of $H(\cdot)$ and $\Phi(\cdot)$, the two cases in the preceding inequality can be combined into
\begin{equation}
G (x) - \Phi (x) \leq - \delta  \Phi (x) + H (x), \label{eq:G-Phi-bound1}
\end{equation}
for all $x \in {\cal X}_n$; recall that $0<\delta<1$ by Definition~\ref{def:GLF}.  Furthermore, the preceding inequality,  $G^b (x) \leq b$ (by definition) and the nonnegativity of $H (\cdot)$ yield
\begin{equation}
G^b (x) - \Phi (x) \wedge b \leq H (x), \label{eq:G-Phi-bound2}
\end{equation}
$x \in {\cal X}_n$; the validity of the inequality can be verified by considering separately the cases $\Phi (x)<b$ and $\Phi (x) \ge b$. Then, (\ref{eq:G-Phi-bound2}) implies
\begin{equation}
\E_{\pi_n}[G^b (\Xi^n_0) - \Phi (\Xi^n_0) \wedge b] \leq \E_{\pi_n}[H (\Xi^n_0)]<\infty, \label{eq:G-Phi-bound3}
\end{equation}
where the strict inequality is due to the choice of $n>n_0$.

Now, the Monotone Convergence Theorem renders $\{G^b (x) - \Phi  (x) \wedge b \}\to \{G (x) - \Phi  (x)\}$ as $b\to\infty$ for every $x\in{\cal X}_n$. Using the Fatou's lemma, applicable due to~(\ref{eq:G-Phi-bound3}) (see also~\cite[p.~44]{CHU74}), we obtain
\begin{align}
\expect_{\pi_n}  \left[ G  (\Xi^n_0)  - \Phi  (\Xi^n_0)  \right] &= \expect_{\pi_n}
\left[ \lim_{b \to \infty} \left\{G ^{b} (\Xi^n_0)  -
\Phi  (\Xi^n_0) \wedge b \right\}\right] \nonumber\\
&\geq \limsup_{b \to \infty} \expect_{\pi_n}  \left[ G ^{b} (\Xi^n_0)  -
\Phi  (\Xi^n_0) \wedge b \right]= 0, \label{eq:EG-Phi}
\end{align}
where the last equality follows from the stationary nature of the distribution $\pi_n$.

Finally, (\ref{eq:G-Phi-bound1}) and~(\ref{eq:EG-Phi}) result in
\[
-\delta  \expect_{\pi_n} \Phi (\Xi^n_0)  +  \expect_{\pi_n} H  (\Xi^n_0) \geq 0,
\]
and the conclusion of the theorem follows since this inequality holds for every $n>n_0$.
\end{proof}

\begin{theorem}\label{theorem:BoundsStationaryPolynomial}
Let  $\{\Xi^n, n\ge 1\}$ be a sequence of discrete-time Markov chains with ${\cal X}_n$ and $\pi_n$ being the state space and a stationary distribution of $\Xi^n$, respectively. Suppose for every $n \geq 1$ function $\Psi: \cup {\cal X}_n \to \rR$ satisfies
\begin{equation}
\expect \left[\Psi^2(\Xi^n_1) \ind{\Xi^n_0 \not\in {\cal R}_n} - \Psi^2(\Xi^n_0) \,\big|\, \Xi_0^n \right] \leq -\delta \Psi(\Xi^n_0) + \psi \label{eq:QLFcond}
\end{equation}
for some $\delta>0$, $0\leq \psi<\infty$ and $\mathcal{R}_n \subset {\cal X}_n$. If
\begin{equation}
C_{\mathcal{R}}\triangleq\limsup_{n\rightarrow \infty}\E_{\pi_n}\left[\Psi^2(\Xi^n_1) \ind{\Xi^n_{0}\in\mathcal{R}_n}\right]<\infty \label{eq:CRdef2}
\end{equation}
and
\begin{equation}
C_{0}\triangleq\limsup_{n\rightarrow \infty}\E_{\pi_n}\left[-\Psi(\Xi^n_0) \ind{\Psi(\Xi^n_{0}) < 0}\right]<\infty \label{eq:C0def}
\end{equation}
then
\[
\limsup_{n\rightarrow \infty} \E_{\pi_n} \Psi(\Xi^n_1) \leq (C_{\mathcal{R}} + C_0 +\psi)/\delta.
\]
\end{theorem}

\begin{remark}
A non-standard part of our definition of the quadratic Lyapunov function is allowing $\Psi$ to be negative. Our second result in this section shows that if a sequence of Markov chains admits the same quadratic Lyapunov function that is uniformly bounded in expectation in the exception region, then the (linear part of this) function is uniformly bounded away from $+\infty$.
\end{remark}

\begin{proof}
The proofs of Theorems~\ref{theorem:BoundsStationaryGeometric} and~\ref{theorem:BoundsStationaryPolynomial} are similar. Assumptions~(\ref{eq:CRdef2}) and~(\ref{eq:C0def}) imply the existence of $n_0$ such that $\E_{\pi_n}[\Psi^2(\Xi^n_1) \ind{\Xi^n_0 \in \mathcal{R}_n} - \Psi(\Xi^n_0) \ind{\Psi(\Xi^n_0)<0}] <\infty$ for all $n>n_0$. Fix an arbitrary such  $n$, introduce the following two conditional expectations
\begin{align*}
G^{b} (x) &\triangleq \expect  \left[\Psi^2(\Xi^n_1)\wedge b \, \big | \,   \Xi^n_0 = x\right], \\
H  (x) &\triangleq \expect \left[\Psi^2 (\Xi^n_1) \ind{\Xi^n_0 \in {\cal R}_{n}} \, \big | \,   \Xi^n_0 = x \right]
\end{align*}
and let $G(x) = G^{\infty}(x)$ for notational simplicity. Then, by~(\ref{eq:QLFcond})
the difference of $G (x)$ and $\Psi^2 (x)$ for $x \in {\cal X}_n$ can be bounded as
\begin{equation}
G (x) - \Psi^2 (x) \leq -\delta  \Psi (x) + \psi + H (x). \label{eq:G-Psi-bound1}
\end{equation}
Furthermore, the preceding inequality,  $G^b (x) \leq b$ (by definition) and the nonnegativity of $H (\cdot)$ yield
\begin{equation}
G^b (x) - \Psi^2 (x) \wedge b \leq -\delta  \Psi (x) \ind{\Psi(x)<0}+ \psi + H (x), \label{eq:G-Psi-bound2}
\end{equation}
$x \in {\cal X}_n$; the validity of the inequality can be verified by considering separately the cases $\Psi^2 (x)<b$ and $\Psi^2 (x) \ge b$. Then, (\ref{eq:G-Psi-bound2}) implies
\begin{equation}
\E_{\pi_n}[G^b (\Xi^n_0) - \Phi (\Xi^n_0) \wedge b] \leq \delta \expect_{\pi_n}\left[-\Psi(\Xi^n_0) \ind{\Psi(\Xi^n_0)<0}\right] + \psi+ \E_{\pi_n}[H (\Xi^n_0)]<\infty, \label{eq:G-Psi-bound3}
\end{equation}
where the strict inequality is due to the choice of $n>n_0$.

Now, the Monotone Convergence Theorem renders $\{G^b (x) - \Psi^2  (x) \wedge b \}\to \{G (x) - \Psi^2  (x)\}$ as $b\to\infty$ for every $x\in{\cal X}_n$. Using the Fatou's lemma, applicable due to~(\ref{eq:G-Psi-bound3}), we obtain
\begin{align}
\expect_{\pi_n}  \left[ G  (\Xi^n_0)  - \Psi^2  (\Xi^n_0)  \right] &= \expect_{\pi_n}
\left[ \lim_{b \to \infty} \left\{G ^{b} (\Xi^n_0)  -
\Psi^2  (\Xi^n_0) \wedge b \right\}\right] \nonumber\\
&\geq \limsup_{b \to \infty} \expect_{\pi_n}  \left[ G ^{b} (\Xi^n_0)  -
\Psi^2  (\Xi^n_0) \wedge b \right]= 0, \label{eq:EG-Psi}
\end{align}
where the last equality follows from the stationary nature of the distribution $\pi_n$.

Finally, (\ref{eq:G-Psi-bound1}) and~(\ref{eq:EG-Psi}) result in
\[
-\delta  \expect_{\pi_n} \Psi (\Xi^n_0)  + \psi +  \expect_{\pi_n} H  (\Xi^n_0) \geq 0,
\]
and the conclusion of the theorem follows since this inequality holds for every $n>n_0$.
\end{proof}

\end{document}